\documentclass[a4paper,10pt
]{amsart}
\usepackage{amsmath,amssymb,amsthm,}
\usepackage[alphabetic, abbrev]{amsrefs}
\usepackage{enumerate}
\usepackage{hyperref}
\usepackage[all]{xy}
\newdir{ >}{{}*!/-5pt/@{>}} 
\SelectTips{cm}{} 

\numberwithin{equation}{section}
\newtheorem{theorem}[subsection]{Theorem}
\newtheorem{corollary}[subsection]{Corollary}
\newtheorem{lemma}[subsection]{Lemma}
\newtheorem{proposition}[subsection]{Proposition}
\theoremstyle{definition}
\newtheorem{definition}[subsection]{Definition}
\newtheorem{notation}[subsection]{Notation}
\newtheorem{remark}[subsection]{Remark}

\newtheorem{construction}[subsection]{Construction}

\newcommand{\bC}{\mathbb{C}}

\newcommand{\bN}{\mathbb{N}}
\newcommand{\bS}{\mathbb{S}}
\newcommand{\bZ}{\mathbb{Z}}
\newcommand{\bF}{\mathbb{F}}

\newcommand{\bm}{\mathbf{m}}

\newcommand{\cC}{\mathcal{C}}

\newcommand{\cJ}{\mathcal{J}}
\newcommand{\cS}{\mathcal{S}}

\newcommand{\longto}{\longrightarrow}

\DeclareMathOperator{\hocolim}{hocolim}

\DeclareMathOperator{\THH}{THH}

\DeclareMathOperator{\cone}{cone}
\DeclareMathOperator{\im}{im}
\DeclareMathOperator{\cok}{cok}

\DeclareMathOperator{\Tor}{Tor}

\renewcommand{\:}{\colon}

\newcommand{\ovl}{\overline}
\newcommand{\ot}{\leftarrow}
\newcommand{\sm}{\wedge}
\newcommand{\iso}{\cong}
\newcommand{\tensor}{\otimes}
\newcommand{\bld}[1]{{\mathbf{#1}}}
\newcommand{\gp}{\mathrm{gp}}
\newcommand{\rep}{\mathrm{rep}}
\newcommand{\cy}{\mathrm{cy}}
\newcommand{\op}{{\mathrm{op}}}
\newcommand{\id}{{\mathrm{id}}}
\newcommand{\Spsym}{{\mathrm{Sp}^{\Sigma}}}

\DeclareMathOperator{\capitalGL}{GL}

\newcommand{\GLoneJof}[1]{\capitalGL^{\cJ}_1\!\!{#1}}
\newcommand{\concat}{\sqcup}
\DeclareMathOperator{\colim}{colim}

\newcommand{\C}[1]{\bC\langle{#1}\rangle}
\newcommand{\postnikovsec}[3]{#1[{#2},{#3}\rangle}

\newcommand{\kup}{ku_{(p)}}
\newcommand{\KUp}{KU_{(p)}}
\newcommand{\AmodmodM}{A/(M_{>0})}

\newcommand{\arxivlink}[1]{\href{http://arxiv.org/abs/#1}{\texttt{arXiv:#1}}}

\title[Logarithmic $\THH$ of topological $K$-theory spectra]{Logarithmic
  topological Hochschild homology\\ of topological $K$-theory spectra}
\author{John Rognes} \address{Department of Mathematics, University of
  Oslo, Box 1053 Blindern, 0316 Oslo,\newline Norway}
\email{rognes@math.uio.no}

\author{Steffen Sagave} \address{IMAPP, Radboud University Nijmegen, PO Box 9010, 6500 GL Nijmegen,\newline The Netherlands} \email{s.sagave@math.ru.nl}

\author{Christian Schlichtkrull}
\address{Department of Mathematics, University of Bergen, PO Box 7803, 5020 Bergen,\newline Norway}
\email{christian.schlichtkrull@math.uib.no}

\date{\today}

\begin{document}
\begin{abstract}
  In this paper we continue our study of logarithmic topological
  Hochschild homology. We show that the inclusion of the connective
  Adams summand~$\ell$ into the $p$-local complex connective
  $K$-theory spectrum~$ku_{(p)}$, equipped with suitable log
  structures, is a formally log $\THH$-{\'e}tale map, and compute the
  $V(1)$-homotopy of their logarithmic topological Hochschild homology
  spectra. As an application, we recover Ausoni's computation of the
  $V(1)$-homotopy of the ordinary $\THH$ of~$ku$.
\end{abstract}

\keywords{Logarithmic ring spectrum, Adams summand, complex topological $K$-theory spectrum}
\subjclass[2010]{55P43; 14F10, 19D55}

\maketitle

\section{Introduction}
Logarithmic topological Hochschild homology (log $\THH$) is an
extension of the usual topological Hochschild homology, which is defined
on more general objects than ordinary rings or ordinary structured
ring spectra. Its input is a \emph{pre-log ring spectrum} $(A,M)$,
consisting of a commutative symmetric ring spectrum $A$ together with certain extra data.
We recall the precise definition below. One reason for considering
this theory is that the log $\THH$ of appropriate pre-log ring spectra
participates in interesting localization homotopy cofiber sequences
that do not exist for ordinary $\THH$. In the first
part~\cite{RSS_LogTHH-I} of this series of papers, we have shown that
there is a localization homotopy cofiber sequence
\begin{equation}\label{eq:loc-sequ-intro} 
\THH(e) \to \THH(e, j_*\!\GLoneJof(E)) \to \Sigma \THH(\postnikovsec{e}{0}{d})
\end{equation}
associated with the connective cover map $j\colon e \to E$ of a
$d$-periodic commutative symmetric ring spectrum $E$. Here $(e,
j_*\!\GLoneJof(E))$ is a pre-log ring spectrum with underlying ring
spectrum $e$, and $ \postnikovsec{e}{0}{d}$ is the $(d-1)$-st
Postnikov section of $e$. Real and complex topological $K$-theory
spectra give rise to examples of this homotopy cofiber sequence.

In the present paper, we compute log $\THH$ in some important
examples. One reason this is interesting is that the homotopy groups
of $\THH(A,M)$ sometimes have a simpler structure than those of
$\THH(A)$. By means of the homotopy cofiber
sequence~\eqref{eq:loc-sequ-intro}, one can then use knowledge about
$\THH(A,M)$ to determine $\THH(A)$.  Specifically, we implement this strategy in the case of the $p$-local complex $K$-theory spectrum $ku_{(p)}$ for an odd prime $p$.  Let $V(1)$ denote the Smith--Toda complex of type 2 (see Notation~\ref{nota:Smith-Toda-complex} below). We shall then first determine the $V(1)$-homotopy of $\THH(ku_{(p)},j_*\!\GLoneJof(KU_{(p)}))$ and based on this compute $V(1)_*\THH(ku_{(p)})$. This realizes the approach to
$V(1)_*\THH(ku)\iso V(1)_*\THH(ku_{(p)})$ outlined by Ausoni
in~\cite{Ausoni_THH-ku}*{\S10} and gives an independent proof of the
main result of~\cite{Ausoni_THH-ku}. One key ingredient for this is
that the tame ramification of the inclusion of the Adams summand $\ell
\to ku_{(p)}$ is detected by log $\THH$. This strategy is motivated by
related results for discrete rings obtained by
Hesselholt--Madsen~\cite{Hesselholt-M_local_fields}*{\S2}. The idea of
extending it to the topological $K$-theory example was first promoted
by Hesselholt.

\subsection{Definition of log \texorpdfstring{$\THH$}{THH}} We briefly recall the definition
of log $\THH$ and refer the reader to~\cite{RSS_LogTHH-I} or Sections
2 to 4 of the present paper for more details, and
to~\cite{RSS_LogTHH-I} and~\cite{Rognes_TLS} for background and
motivation. Let $A$ be a commutative symmetric ring spectrum.  It has
an underlying \emph{graded} multiplicative $E_{\infty}$~space~$\Omega^{\cJ}(A)$.  This object $\Omega^{\cJ}(A)$ is a
\emph{commutative $\cJ$-space monoid} in the sense of~\cite[Section
4]{Sagave-S_diagram}, i.e., a symmetric monoidal functor from a
certain indexing category $\cJ$ to spaces.  A pre-log ring spectrum
$(A,M)$ is a commutative symmetric ring spectrum $A$ together with a
commutative $\cJ$-space monoid $M$ and a map of commutative
$\cJ$-space monoids $\alpha \colon M \to \Omega^{\cJ}(A)$. The
following direct image construction is a source of non-trivial pre-log
ring spectra: If $j\colon e \to E$ is the connective cover map of a
periodic commutative symmetric ring spectrum $E$, then one can form a
diagram of commutative $\cJ$-space monoids
\begin{equation}\label{eq:jGLoneJofE-intro} 
\GLoneJof(E) \to \Omega^{\cJ}(E) \ot \Omega^{\cJ}(e), 
\end{equation}
where $\GLoneJof(E)$ is the commutative $\cJ$-space monoid of
\emph{graded} units associated with the ring spectrum $E$.  The
pullback $j_*\!\GLoneJof(E)$ of~\eqref{eq:jGLoneJofE-intro} comes with a
canonical map to $\Omega^{\cJ}(e)$ and defines a pre-log ring spectrum
$(e, j_*\!\GLoneJof(E))$.

Let $(A,M)$ be a pre-log ring spectrum. By definition, its logarithmic
topological Hochschild homology $\THH(A,M)$ is the homotopy pushout of
the following diagram of cyclic commutative symmetric ring spectra:
\begin{equation}\label{eq:pushout-logTHH-intro} 
 \bS^{\cJ}[B^{\rep}(M)] \ot \bS^{\cJ}[B^{\cy}(M)] \to \THH(A). 
\end{equation} The
right hand term is the ordinary $\THH$ of $A$, given by the cyclic bar
construction on $A$. The middle term is the graded suspension  
spectrum associated with the cyclic bar construction $B^{\cy}(M)$ on
the commutative $\cJ$-space monoid $M$. The right hand map is induced
by the adjoint $\bS^{\cJ}[M] \to A$ of the structure map $\alpha \colon M \to
\Omega^{\cJ}(A)$ of $(A,M)$.  The left hand map is induced by the map
$B^{\cy}(M) \to B^{\rep}(M)$ to the \emph{replete} bar construction
$B^{\rep}(M)$ of $M$. The latter can be viewed as a variant of
$B^{\cy}(M)$ that is formed relative to the group completion of $M$.
\subsection{Log \texorpdfstring{$\THH$}{THH} of the Adams summand}
Let $p$ be an odd prime and let $\ell$ be the Adams summand of the
$p$-local complex connective $K$-theory spectrum $ku_{(p)}$. It is
known that the map $j\colon \ell \to L$ to the periodic version $L$ of
$\ell$ can be represented by a map of commutative symmetric ring
spectra. Hence we can form the pre-log ring spectrum $(\ell,
j_*\!\GLoneJof(L))$. In this case, the homotopy cofiber
sequence~\eqref{eq:loc-sequ-intro} relates $\THH( \ell,
j_*\!\GLoneJof(L))$ to the ordinary $\THH$ of $\ell$ and $H\bZ_{(p)}$.

Writing $E$ and $P$ for exterior and polynomial algebras over $\bF_p$,
respectively, we can formulate our first main result.
\begin{theorem}\label{thm:V(1)THHellGLoneJofL-intro}
  There is an algebra isomorphism
  \[
  V(1)_* \THH(\ell, j_*\!\GLoneJof(L)) \iso E(\lambda_1, d\log v)
  \otimes P(\kappa_1) \,,
  \]
  with $|\lambda_1| = 2p-1$, $|d\log v| = 1$ and $|\kappa_1| =
  2p$. The suspension operator~$\sigma$ arising from the circle action on $\THH(\ell, j_*\!\GLoneJof(L))$ satisfies $\sigma(\kappa_1) = \kappa_1 \cdot
  d\log v$, and is zero on $\lambda_1$ and $d\log v$.
\end{theorem}
The strategy for the proof of
Theorem~\ref{thm:V(1)THHellGLoneJofL-intro} is as follows.  In a first
step, we use the invariance of log $\THH$ under logification
established in \cite[Theorem~4.24]{RSS_LogTHH-I} to replace
$(\ell, j_*\!\GLoneJof(L))$ by a pre-log ring spectrum $(\ell, D(v))$
with equivalent log $\THH$. This $(\ell, D(v))$ was also considered
in~\cite{Sagave_log-on-k-theory}. Its advantage is that the
$E_{\infty}$ space $\hocolim_{\cJ}D(v)$ associated with $D(v)$ is
equivalent to $Q_{\geq 0}S^0$, the non-negative components of $QS^0
= \Omega^{\infty}\Sigma^{\infty}S^0$. So the homology of
$\hocolim_{\cJ}D(v)$ is well understood and independent of
$\ell$. Using the graded Thom isomorphism established by the last two
authors in~\cite{Sagave-S_graded-Thom}, this allows us to determine
the homology of $\bS^{\cJ}[D(v)]$, $\bS^{\cJ}[B^{\cy}(D(v))]$, and
$\bS^{\cJ}[B^{\rep}(D(v))]$. Combining this with the computation of
$V(1)_* \THH(\ell)$ by McClure and Staffeldt~\cite{McClure-S_thh-bu},
an application of the K{\"u}nneth spectral sequence associated with
the homotopy pushout~\eqref{eq:pushout-logTHH-intro} leads to our
computation of $V(1)_* \THH(\ell, j_*\!\GLoneJof(L))$.

\subsection{The inclusion of the Adams summand}
In analogy with the notion of \emph{formally $\THH$-{\'e}tale} maps
in~\cite[\S 9.2]{Rognes_Galois}, we say that $(A,M) \to (B,N)$ is a
\emph{formally log $\THH$-{\'e}tale} map of pre-log ring spectra if
$B\sm_A \THH(A,M) \to \THH(B,N) $ is a stable equivalence. Our second
main theorem verifies this property in an example:
\begin{theorem}\label{thm:ell-ku-direct-image-logthh-etale-intro}
The inclusion of the Adams summand $\ell \to ku_{(p)}$ induces a
formally log $\THH$-{\'e}tale map $(\ell, j_*\!\GLoneJof(L)) \to
(ku_{(p)}, j_*\!\GLoneJof(KU_{(p)}))$.
\end{theorem}
Here $KU_{(p)}$ is the periodic version of $ku_{(p)}$, and
$j_*\!\GLoneJof(KU_{(p)})$ denotes the direct image pre-log structure
on $ku_{(p)}$, constructed as above. We stress that the proof of this
theorem does not depend on computations of $ \THH(\ell,
j_*\!\GLoneJof(L)) $ and $\THH(\kup,j_*\!\GLoneJof(KU_{(p)}))$. Instead,
it is based on a certain decomposition of the replete bar construction,
and on the graded Thom isomorphism and the invariance of log $\THH$
under logification mentioned above.

It is shown in~\cite[Theorem 1.6]{Sagave_log-on-k-theory} that $(\ell,
j_*\!\GLoneJof(L)) \to (ku_{(p)}, j_*\!\GLoneJof(KU_{(p)}))$ is also
formally \'{e}tale with respect to logarithmic topological
Andr\'{e}-Quillen homology. Since the logarithmic K{\"a}hler
differentials of algebraic geometry can be used to measure
ramification beyond \emph{tame} ramification of discrete valuation
rings, these results show that $\ell \to ku_{(p)}$ behaves as a tamely
ramified extension of ring spectra. By analogy with Emmy Noether's
correspondence between tame ramification and the existence of normal
bases, these results are compatible with the fact that $ku_{(p)}$ is a
retract of a finite cell $\ell[\Delta]$-module spectrum, where $\Delta
= (\bZ/p)^\times$ is the Galois group of $L \to KU_{(p)}$.

\subsection{\texorpdfstring{$\THH$}{THH} of the connective complex \texorpdfstring{$K$}{K}-theory spectrum}
Combining the last two theorems leads to the following result. Here $P_{p-1}$ denotes a height $p-1$ truncated polynomial algebra. 
\begin{theorem}\label{thm:V(1)THHkuGLoneJofKU-intro}
There is an algebra isomorphism
\[
V(1)_* \THH(\kup,j_*\!\GLoneJof(KU_{(p)})) \cong P_{p-1}(u) \otimes E(\lambda_1, d\log u)
        \otimes P(\kappa_1)
\]
with $|u| = 2$, $|\lambda_1| = 2p-1$, $|d\log u| = 1$ and $|\kappa_1|
= 2p$. The suspension operator~$\sigma$ arising from the circle action
on $\THH(\kup,j_*\!\GLoneJof(KU_{(p)}))$ satisfies $\sigma(u) = {u
\cdot d\log u}$ and $\sigma(\kappa_1) = -\kappa_1 \cdot d\log u$, and
is zero on $\lambda_1$ and $d\log u$.
\end{theorem}
In logarithmic algebraic geometry, the passage from K{\"a}hler
differentials to logarithmic K{\"a}hler differentials allows one to have
differentials with logarithmic poles, i.e., it introduces elements
$d\log x$ satisfying $x\cdot d\log x = dx$. By analogy with the
Hochschild--Kostant--Rosenberg correspondence between $(\mathrm{HH},\sigma)$
and $(\Omega,d)$, one may expect similar phenomena for logarithmic
$\THH$. In view of this, the above relation $\sigma(u) = u \cdot d\log u$
is a justification for denoting the relevant homotopy class by $d\log
u$.

Using the homotopy cofiber sequence~\eqref{eq:loc-sequ-intro}, the
previous theorem allows us to recover Ausoni's computation of the
rather complicated finitely presented $\bF_p$-algebra
$V(1)_*\THH(ku_{(p)})$~\cite{Ausoni_THH-ku}. For this application of
Theorem~\ref{thm:V(1)THHkuGLoneJofKU-intro} it is important that
already in the case of the Adams summand, the explicit definition of
logarithmic $\THH$ allows us to determine the homomorphisms in the
long exact sequence of $V(1)$-homotopy groups induced
by~\eqref{eq:loc-sequ-intro}. It is not clear if the construction of a
localization homotopy cofiber sequence for $\THH$ by
Blumberg--Mandell~\cite{Blumberg-M_loc-sequenceTHH} provides such an
explicit understanding of the resulting long exact
sequence. Nonetheless, we expect our sequence to be equivalent to
theirs, in the special cases they consider. If true, this would be one
way to relate our homotopy cofiber sequence~\eqref{eq:loc-sequ-intro}
to the corresponding $K$-theoretical localization sequence.

\subsection{Organization}
In Section~\ref{sec:cyclic-bar} we briefly review commutative
$\cJ$-space monoids and their cyclic bar construction and prove a
decomposition result for the cyclic bar construction of grouplike
commutative $\cJ$-space monoids. In Section~\ref{sec:rep-bar} we
review the replete bar construction and prove a similar decomposition
formula. In Section~\ref{sec:log-thh} we briefly recall the definition
of log $\THH$. In Section~\ref{sec:graded-Thom} we explain how the
graded Thom isomorphism established in~\cite{Sagave-S_graded-Thom} can
be used to compute the homology of $\bS^{\cJ}[M]$ for certain commutative
$\cJ$-space monoids $M$.  Section~\ref{sec:elog-etaleness} contains
the proof of
Theorem~\ref{thm:ell-ku-direct-image-logthh-etale-intro}. In
Section~\ref{sec:log-THH-ell} we compute the $V(1)$-homotopy of the
log $\THH$ of the Adams summand and prove
Theorem~\ref{thm:V(1)THHellGLoneJofL-intro}. In the final
Section~\ref{sec:log-THH-ku} we prove
Theorem~\ref{thm:V(1)THHkuGLoneJofKU-intro} about the log $\THH$ of
$\kup$ and explain how to use this for computing the $V(1)$-homotopy
of $\THH(ku_{(p)})$.
\subsection{Acknowledgments} The authors would like to thank the
  referee for useful comments.

\section{The cyclic bar construction for commutative \texorpdfstring{$\cJ$}{J}-space monoids}\label{sec:cyclic-bar}
We briefly recall some terminology that is needed to state the definition of
logarithmic $\THH$ in Section~\ref{sec:log-thh}. More details on these
foundations can be found in~\cite[Section~4]{Sagave-S_diagram}.

\subsection{\texorpdfstring{$\cJ$}{J}-spaces}
Let $\cJ$ be the category given by Quillen's localization construction
on the permutative category $\Sigma$ of finite sets and bijections. It
is a symmetric monoidal category whose classifying space $B\cJ$ is
weakly equivalent to $QS^0$. The objects of $\cJ$ are pairs
$(\bld{m_1},\bld{m_2})$ of finite sets of the form $\bld{m_i} =
\{1,\dots, m_i\}$, where each $m_i \geq 0$.  A \emph{$\cJ$-space} is a
functor from $\cJ$ to simplicial sets. For each object
$(\bld{m_1},\bld{m_2})$ of $\cJ$ there is a functor from the category
$\cS$ of simplicial sets to $\cJ$-spaces \[F_{(\bld m_1,\bld
  m_2)}^{\cJ}\colon \cS \to \cS^{\cJ}, \qquad K \mapsto \cJ((\bld
m_1,\bld m_2),-) \times K,\] which is left adjoint to the evaluation of
a $\cJ$-space at $(\bld{m_1},\bld{m_2})$.

The category of $\cJ$-spaces admits a
Day type convolution product $\boxtimes$ induced by the ordered
concatenation of finite sets and the cartesian product of simplicial
sets. The unit $U^{\cJ}$ of this symmetric monoidal product is the
functor $\cJ((\bld{0},\bld{0}),-)$, corepresented by the monoidal unit
$(\bld{0},\bld{0})$ of $\cJ$. A \emph{commutative $\cJ$-space monoid}
is a commutative monoid object in $(\cS^{\cJ},\boxtimes,U^{\cJ})$, and
we write $\cC\cS^{\cJ}$ for the category of commutative $\cJ$-space
monoids.

The category $\cC\cS^{\cJ}$ admits a proper simplicial positive
projective model structure where $M \to N$ is a weak equivalence if it
induces a weak equivalence of spaces $M_{h\cJ} \to
N_{h\cJ}$~\cite[Proposition 4.10]{Sagave-S_diagram}. Here $M_{h\cJ} =
\hocolim_{\cJ} M$ denotes the Bousfield--Kan homotopy colimit of $M$
over $\cJ$, which is an associative (but not commutative)
simplicial monoid.  In the following we will refer to this model
structure as the \emph{positive $\cJ$-model} structure and call its
weak equivalences the \emph{$\cJ$-equivalences}. Unless otherwise
stated, the notions of cofibrations or fibrations in $\cC\cS^{\cJ}$
will refer to this model structure. Equipped with the positive
$\cJ$-model structure, $\cC\cS^{\cJ}$ is Quillen equivalent
to the category of $E_{\infty}$ spaces over $B\cJ$.  We therefore
think of commutative $\cJ$-space monoids as ($QS^0$-)graded
$E_{\infty}$ spaces. The category $\cJ$ is closely related to
symmetric spectra. In particular, there is a Quillen adjunction
\[ 
\bS^{\cJ} \colon \cC\cS^{\cJ} \rightleftarrows \cC\Spsym \colon \Omega^{\cJ}
\] 
relating $\cC\cS^{\cJ}$ to the category of commutative symmetric ring
spectra $\cC\Spsym$ with the positive projective stable model
structure.  If $A$ is a commutative symmetric ring spectrum, we view
$\Omega^{\cJ}(A)$ as a model for the underlying graded multiplicative
$E_{\infty}$ space of $A$.  We say that a commutative $\cJ$-space
monoid $M$ is \emph{grouplike} if the monoid $\pi_0(M_{h\cJ})$ is a
group. If $A$ is positive fibrant, then $\Omega^{\cJ}(A)$ has a
subobject $\GLoneJof(A)$ of \emph{graded units} such that inclusion
$\GLoneJof(A) \to \Omega^{\cJ}(A)$ corresponds to the inclusion
$\pi_*(A)^{\times} \to \pi_*(A)$ of the units of the graded ring of
stable homotopy groups of $A$.
\subsection{The cyclic bar construction}
Let $M$ be a commutative $\cJ$-space monoid. The \emph{cyclic bar
  construction} $B^{\cy}(M)$ of $M$ is the realization of a
simplicial object $[q] \mapsto M^{\boxtimes (q+1)}$.  Its structure maps
are defined using the unit and the multiplication of $M$ and the twist
isomorphism for the symmetric monoidal product $\boxtimes$\,;
see~\cite[Section~3]{RSS_LogTHH-I} for
details. The object $B^{\cy}(M)$ will be one of the three building
blocks of log $\THH$. We note that since $M$ is
commutative, the iterated multiplication maps induce a natural
augmentation $\epsilon\colon B^{\cy}(M)\to M$.

Our first goal is to decompose $B^{\cy}(M)$ as a coproduct of
commutative $\cJ$-space monoids, in the case when $M$ is
grouplike. For this we fix a factorization of
the unit of $M$ into an acyclic cofibration followed by a positive
$\cJ$-fibration as indicated in the bottom row of the diagram
\[
\xymatrix@-1pc{
&& V(M) \ar@{->>}[rr] \ar[d]&& B^{\cy}(M)\ar[d]^{\epsilon}\\
U^{\cJ} \ar@{>->}[rr]^-{\sim} &&U(M)\ar@{->>}[rr] && M. 
}
\]
We define $V(M)$ to be the pullback of the augmentation
$\epsilon\colon B^{\cy}(M)\to M$ and $U(M) \to M$. It is a model for
the homotopy fiber of the augmentation over the unit. Using the
multiplicative structure of $B^{\cy}(M)$ we get a map of commutative
$\cJ$-space monoids
\[
M\boxtimes V(M) \to B^{\cy}(M)\boxtimes B^{\cy}(M) \to B^{\cy}(M).
\]
\begin{proposition}\label{prop:grouplike-M-VM-equivalence}
The map $M\boxtimes V(M)\to B^{\cy}(M)$ is a $\cJ$-equivalence provided that $M$ is grouplike and cofibrant.
\end{proposition}
\begin{proof}
Consider the commutative diagram of homotopy colimits
\[
\xymatrix@-1pc{
V(M)_{h\cJ} \ar[r] \ar@{=}[d]& M_{h\cJ}\times V(M)_{h\cJ} \ar[r] \ar[d]_{\sim}& 
M_{h\cJ}\times U(M)_{h\cJ}\ar[d]_{\sim}\\
V(M)_{h\cJ} \ar[r] \ar@{=}[d] & (M\boxtimes V(M))_{h\cJ} \ar[r]\ar[d] & (M\boxtimes U(M))_{h\cJ}\ar[d]_{\sim}\\ 
V(M)_{h\cJ} \ar[r] & B^{\cy}(M)_{h\cJ} \ar[r] & M_{h\cJ} 
}
\]
in which the map in the lemma induces the middle lower vertical
arrow. The bottom part is obtained by passing to homotopy colimits from
the corresponding diagram of commutative $\cJ$-space monoids, and the
vertical equivalences in the upper part of the diagram arise from the
monoidal structure map of the homotopy colimit.

We must show that the vertical composition $M_{h\cJ}\times V(M)_{h\cJ}
\to B^{\cy}(M)_{h\cJ}$ is a weak homotopy equivalence. Notice that the
latter is equivariant as a map of spaces with left
$M_{h\cJ}$-action. Furthermore, the assumption that $M$ is commutative
and grouplike implies that $V(M)_{h\cJ}$ is path connected and that
$M_{h\cJ}\to B^{\cy}(M)_{h\cJ}$ induces an isomorphism of path
components. Since $M_{h\cJ}$ is assumed to be grouplike it therefore
suffices to show that the map in question restricts to a weak homotopy
equivalence on the path components containing the unit elements. For
this we observe that the corresponding restriction of the above
diagram is a diagram of horizontal homotopy fiber sequences, which
gives the result.
\end{proof}
The next proposition identifies the homotopy type of $V(M)$
for grouplike $M$.
\begin{proposition}\label{prop:grouplike-V(M)-equivalence}
  Suppose that $M$ is a grouplike and cofibrant commutative
  $\cJ$-space monoid. Then there is a chain of natural
  $\cJ$-equivalences of $\cJ$-spaces augmented over $M$ relating
  $V(M)$ and $U^{\cJ}\times B(M_{h\cJ})$.
\end{proposition}
The proof of this proposition needs some preparation and will be given
at the end of this section.

\subsection{The bar resolution of \texorpdfstring{$\cJ$}{J}-spaces} 
Let $X$ be a $\cJ$-space. We define the \emph{bar
  resolution} $\overline X$ of $X$ to be the $\cJ$-space given by the
bar construction
\[
\overline X(\bld n_1,\bld n_2)=B(\cJ(-,(\bld n_1,\bld n_2)),\cJ,X),
\] 
where we view $\cJ(-,(\bld n_1,\bld n_2))$ as a $\cJ^{\op}$-space in
the obvious way. (See e.g.\ \cite{Hollender-V_modules} for a
discussion of the bar construction in the context of diagram spaces.)
By definition, this is the same as the homotopy left Kan extension of
$X$ along the identity functor on $\cJ$. Equivalently, we may describe
$\overline X(\bld n_1,\bld n_2)$ as the homotopy colimit
\[
\overline X(\bld n_1,\bld n_2)=\hocolim_{\cJ\downarrow(\bld n_1,\bld n_2)}X\circ \pi_{(\bld n_1,\bld n_2)}
\]
over the comma category $\cJ\downarrow(\bld n_1,\bld n_2)$ of
objects in $\cJ$ over $(\bld n_1,\bld n_2)$. Here $\pi_{(\bld n_1,\bld
  n_2)}$ denotes the forgetful functor from $\cJ\downarrow(\bld
n_1,\bld n_2)$ to $\cJ$. Each of the categories $\cJ\downarrow(\bld
n_1,\bld n_2)$ has a terminal object and hence the projection of the
homotopy colimit onto the colimit defines an evaluation map of
$\cJ$-spaces $\overline X\to X$ that is a level equivalence.

\begin{lemma}\label{lem:colim-bar-resolution}
There is a natural isomorphism $\colim_{\cJ}\overline X\cong X_{h\cJ}$. \qed
\end{lemma}
As a consequence of the lemma there is a natural map of $\cJ$-spaces
$\overline X\to X_{h\cJ}$ when we view $X_{h\cJ}$ as a constant
$\cJ$-space. (Notice that this is not a $\cJ$-equivalence since $B\cJ$
is not contractible). The lemma suggests that $\overline X$
is a kind of cofibrant replacement of $X$. More precisely we have the
following result, which can be deduced from the skeletal filtration of
the bar construction.

\begin{lemma}\label{lem:bar-resolution-cofibrant}
  Let $X$ be a $\cJ$-space. As a $\cJ$-space, the bar resolution
  $\overline X$ is cofibrant in the absolute projective model
  structure of~\cite[Proposition 4.8]{Sagave-S_diagram}.\qed
\end{lemma}

Clearly the bar resolution $X\mapsto \overline X$ is functorial in $X$
and we claim that it canonically has the structure of a lax monoidal
functor. Indeed, the monoidal product $\overline X\boxtimes \overline
Y\to \overline{X\boxtimes Y}$ is induced by the natural map of
$\cJ\times\cJ$-spaces
\[
\begin{aligned}
&B(\cJ(-,(\bld m_1,\bld m_2)),\cJ,X)\times B(\cJ(-,(\bld n_1,\bld n_2)),\cJ,Y)\\
&\cong B(\cJ(-,(\bld m_1,\bld m_2))\times \cJ(-,(\bld n_1,\bld n_2)),\cJ\times\cJ,X\times Y)\\
&\to B(\cJ(-,(\bld m_1,\bld m_2)\sqcup(\bld n_1,\bld n_2)),\cJ,X\boxtimes Y).
\end{aligned}
\] 
Here we use that the bar construction preserves products. The second
map is induced by the monoidal structure map
$\sqcup\colon\cJ\times\cJ\to \cJ$, the canonical map of
$\cJ\times\cJ$-spaces $X\times Y\to \sqcup^*(X\boxtimes Y)$, and the
map of $(\cJ\times\cJ)^{\op}$-spaces in the first variable determined
by $\sqcup$. The monoidal unit is the unique map of $\cJ$-spaces from
the unit $U^{\cJ}$ for the $\boxtimes$-product to its bar
resolution. Furthermore, it is easy to check that the evaluation
$\overline X\to X$ is a monoidal natural transformation. This implies
that the bar resolution of a $\cJ$-space monoid $M$ is again a
$\cJ$-space monoid and that $\overline M\to M$ is a map of $\cJ$-space
monoids. By Lemma~\ref{lem:colim-bar-resolution} there also is a
natural map of $\cJ$-space monoids $\overline M\to M_{h\cJ}$.

\begin{remark}
  The bar resolution functor is not lax symmetric monoidal, and
  consequently it does not take commutative $\cJ$-space monoids to
  commutative $\cJ$-space monoids. 
\end{remark}

\begin{lemma}\label{lem:bar-cyclic-equivalence}
  Let $M$ be a cofibrant commutative $\cJ$-space monoid. Then
  evaluation $\overline M\to M$ induces a level equivalence
  $B^{\cy}(\overline M)\to B^{\cy}(M)$.
\end{lemma} 
\begin{proof}
  By Lemma~\ref{lem:bar-resolution-cofibrant} above
  and~\cite[Proposition~4.28]{Sagave-S_diagram}, the underlying
  $\cJ$-spaces of $\overline M$ and $M$ are flat in the sense of
  ~\cite[Section 4.27]{Sagave-S_diagram}. Since $\overline M\to M$ is
  a level equivalence, \cite[Proposition~8.2]{Sagave-S_diagram} implies that
  $B^{\cy}_{q}(\overline M)\to B^{\cy}_{q}(M)$ is a level equivalence
  in every simplicial degree $q$. The claim follows by the realization
  lemma for bisimplicial sets.
\end{proof}
We now use the bar resolution to analyze the homotopy colimit of
$B^{\cy}(M)$ under suitable assumptions on $M$. For this we note that
one can also define the cyclic bar construction $B^{\cy}$ in
$(\cS,\times,*)$ and apply it to associative simplicial monoids.

\begin{lemma}\label{lem:Bcy-bar-hocolim}
  There is a natural weak equivalence $B^{\cy}(\overline M)_{h\cJ}\to
  B^{\cy}(M_{h\cJ})$.
\end{lemma}
\begin{proof}
  Notice first that $B^{\cy}(\overline M)_{h\cJ}$ is isomorphic to the
  realization of the cyclic space $B^{\cy}_{\bullet}(\overline
  M)_{h\cJ}$ obtained by evaluating the homotopy colimit in each
  simplicial degree.  Now we use that the colimit functor from
  $\cS^{\cJ}$ to $\cS$ is strong symmetric monoidal, with respect to
  the $\boxtimes$-product on $\cS^{\cJ}$ and the usual categorical
  product on $\cS$, to get a natural map of cyclic spaces
  \begin{equation}\label{eq:Bcy-bar-hocolim}
  \hocolim_{\cJ}B_{\bullet}^{\cy}(\overline M)\to
  \colim_{\cJ}B_{\bullet}^{\cy}(\overline M)\cong
  B_{\bullet}^{\cy}(\colim_{\cJ}\overline M)\cong
  B_{\bullet}^{\cy}(M_{h\cJ}).
  \end{equation}
  By the realization lemma for bisimplicial sets it is enough to show
  that the first map is a weak homotopy equivalence in each simplicial
  degree. Since $\overline M$ and hence $B_{q}^{\cy}(\overline M)$ 
  are cofibrant $\cJ$-spaces, this follows from~\cite[Lemma~6.22]{Sagave-S_diagram}.
\end{proof}

\begin{corollary}\label{cor:BcyM-hJ-vs-Bcy-of-MhJ}
For a cofibrant commutative $\cJ$-space monoid $M$ there is a chain of natural weak homotopy equivalences 
\[
B^{\cy}(M)_{h\cJ}\xleftarrow{\sim} B^{\cy}(\overline M)_{h\cJ} \xrightarrow{\sim} B^{\cy}(M_{h\cJ}).
\eqno\qed
\]
\end{corollary}

In the next lemma we consider the classifying space $B(M_{h\cJ})$ and
the levelwise cartesian product $M\times B(M_{h\cJ})$. The latter may
be interpreted as either the tensor of $M$ with the space
$B(M_{h\cJ})$, the $\boxtimes$-product of $M$ with $F_{(\bld 0,\bld
  0)}^{\cJ}(B(M_{h\cJ}))$, or the cartesian product of $M$ with the
constant $\cJ$-space defined by $B(M_{h\cJ})$.

\begin{lemma}\label{lem:Bcy-bar-M-decomposition}
  Let $M$ be a commutative $\cJ$-space monoid. There is a natural
  map of $\cJ$-spaces $(\ovl{\epsilon},\ovl{\pi})\colon B^{\cy}(\overline M)\to M\times B(M_{h\cJ})$,
  which is a $\cJ$-equivalence if $M$ is grouplike.
\end{lemma}
\begin{proof}
  Consider the map of $\cJ$-spaces $B^{\cy}(\overline M)\to
  B^{\cy}(M)\to M$ induced by the evaluation $\overline M\to M$ and
  the augmentation of $B^{\cy}(M)$, using that $M$ is
  commutative. This gives the first factor $\ovl{\epsilon}$ of the map in the
  lemma. The second factor $\ovl{\pi}$ is defined by the composition
  $B^{\cy}(\overline M)\to \mathrm{const}_{\cJ}B^{\cy}(M_{h\cJ})\to
  \mathrm{const}_{\cJ}B(M_{h\cJ})$ where the first map is the adjoint
  of the isomorphism $\colim_{\cJ}B^{\cy}(\overline M) \to
  B^{\cy}(M_{h\cJ})$ used in the definition of the
  map~\eqref{eq:Bcy-bar-hocolim} and the second map is given by the
  projection away from the zeroth coordinate in each simplicial degree.
  The induced map of homotopy colimits $B^{\cy}(\overline M)_{h\cJ}
  \to (M\times B(M_{h\cJ}))_{h\cJ}\iso M_{h\cJ} \times B(M_{h\cJ})$
  fits into a commutative diagram
  \[
  \xymatrix@-1pc{
    M_{h\cJ} \ar[r] & B^{\cy}(M_{h\cJ}) \ar[r] & B(M_{h\cJ})\\
    \overline{M}_{h\cJ} \ar[r]\ar[u]^{\sim} \ar[d]_{\sim}& B^{\cy}(\overline{M})_{h\cJ} \ar[r]\ar[u]^{\sim} \ar[d]& B(M_{h\cJ})\ar@{=}[u] \ar@{=}[d]\\
    M_{h\cJ} \ar[r] & M_{h\cJ}\times B(M_{h\cJ}) \ar[r] & B(M_{h\cJ})
  }
  \]
  as the middle lower vertical arrow. The equivalences
  $\overline{M}_{h\cJ}\to~M_{h\cJ}$ are defined as follows: In the
  lower part of the diagram it is induced by the evaluation
  $\overline M\to M$, whereas in the upper part of the diagram it is
  given by the projection from the homotopy colimit to the colimit
  using the identification in Lemma~\ref{lem:colim-bar-resolution},
  see also \cite[Theorem~5.5]{Hollender-V_modules}. (These two
  equivalences are canonically homotopic but that is not relevant for
  the argument.) The assumption that $M$ is grouplike implies that the
  upper row is a homotopy fiber sequence in the sense that the map
  from $M_{h\cJ}$ to the homotopy fiber of the second map is a weak
  homotopy equivalence. This follows from standard results on
  geometric realization of simplicial quasifibrations as in the proof
  of \cite[Lemma~V.I.3]{Goodwillie-cyclic}. Hence the middle row is
  also a homotopy fiber sequence by Lemma~\ref{lem:Bcy-bar-hocolim},
  which gives the result.
\end{proof}

\begin{proof}[Proof of Proposition~\ref{prop:grouplike-V(M)-equivalence}]
  Let $M$ be a grouplike and cofibrant commutative $\cJ$-space monoid
  and let $\overline{V}(M)$ be defined as the pullback of the diagram
  \[\xymatrix{U(M) \ar@{->>}[r] & M & B^{\cy}(\overline M)
    \ar[l]_-{\ovl{\epsilon}}},\]
  where the map on the right is defined as in
  Lemma~\ref{lem:Bcy-bar-M-decomposition}.  We claim that there is a
  chain of natural $\cJ$-equivalences of $\cJ$-spaces over $M$
  \begin{equation}\label{eq:grouplike-V(M)-equivalence}
  V(M)\xleftarrow{\sim} \overline{V}(M) \xrightarrow{\sim} U(M)\times
  B(M_{h\cJ})\xleftarrow{\sim} U^{\cJ}\times B(M_{h\cJ}).
  \end{equation}
  To obtain the maps in~\eqref{eq:grouplike-V(M)-equivalence}, we note
  that the evaluation map $\overline M\to M$ induces a
  $\cJ$-equivalence $B^{\cy}(\overline M)\to B^{\cy}(M)$ by
  Lemma~\ref{lem:bar-cyclic-equivalence}. There is an induced
  $\cJ$-equivalence $\overline V(M)\to V(M)$ since the positive
  $\cJ$-model structure is right proper. For the second equivalence we
  use the given map to $U(M)$ in the first factor and the second
  factor is the composition
  \[
  \overline V(M) \to B^{\cy}(\overline M) \xrightarrow{\ovl{\pi}} \mathrm{const}_{\cJ} B(M_{h\cJ})
  \]
  of the given map to $B^{\cy}(\overline M)$ with the map from
  Lemma~\ref{lem:Bcy-bar-M-decomposition}.  Now consider the homotopy
  cartesian square of $\cJ$-spaces
  \[
  \xymatrix@-1pc{
    U(M)\times B(M_{h\cJ}) \ar@{->>}[rr]\ar[d] && M\times B(M_{h\cJ})\ar[d]\\
    U(M) \ar@{->>}[rr] && M.  }
  \]
  Together with the $\cJ$-equivalence from Lemma
  \ref{lem:Bcy-bar-M-decomposition} the map just described defines a
  map from the square defining $\overline V(M)$ to the latter
  square. Hence the result again follows from right properness of the
  positive $\cJ$-model structure.
\end{proof}

\section{The replete bar construction for commutative \texorpdfstring{$\cJ$}{J}-space monoids}\label{sec:rep-bar}
Let $M$ be a commutative $\cJ$-space monoid.  As the second building
block of the logarithmic $\THH$ to be defined in
Section~\ref{sec:log-thh}, we now recall the definition of the
\emph{replete bar construction} $B^{\rep}(M)$ from~\cite[Section~3.3]{RSS_LogTHH-I}. Let $M \to M^{\gp}$ be 
a chosen functorial group completion of $M$ and let
\begin{equation}\label{eq:factorization-for-BrepM-def}
\xymatrix{M \ar@{ >->}[r]^{\sim}&M' \ar@{->>}[r] & M^{\gp}}\end{equation} be a (functorial)
factorization of $M \to M^{\gp}$ into an acyclic cofibration followed
by a fibration. The replete bar construction $B^{\rep}(M)$  is defined as the pullback
of the diagram of commutative $\cJ$-space monoids 
\[\xymatrix{ M' \ar@{->>}[r]& M^{\gp} & B^{\cy}(M^{\gp}) \ar[l]_-{\epsilon}}\] provided by the above map
$M' \to M^{\gp}$ and the augmentation $\epsilon\colon B^{\cy}(M^{\gp}) \to
M^{\gp}$. By construction, there is a natural \emph{repletion} map
$\rho\colon B^{\cy}(M) \to B^{\rep}(M)$. 
\begin{proposition}\label{prop:BrepM-simeq-MtimesBMhJ}
  Let $M$ be a cofibrant commutative $\cJ$-space monoid and view
  $M^{\gp}$ as a left $M$-module via $M\to M^{\gp}$. There is a chain
  of natural $\cJ$-equivalences of left $M$-modules over $M^{\gp}$
  relating $B^{\rep}(M)$ and $M\times B(M_{h\cJ})$.
\end{proposition}
\begin{proof}
  Given a factorization $M \to M' \to M^{\gp}$ as
  in~\eqref{eq:factorization-for-BrepM-def}, we factor the unit
  map $U^{\cJ} \to M'$ as an acyclic cofibration $U^{\cJ}\to U(M^{\gp})$
  followed by a fibration $ U(M^{\gp}) \to M'$. This  provides the lower part
  of the following commutative diagram:
 \[
  \xymatrix@-1pc{
   V(M^{\gp}) \ar[rr] \ar[d] & & B^{\rep}(M)\ar[rr] \ar[d] & & B^{\cy}(M^{\gp})\ar[d] \\
   U(M^{\gp}) \ar@{->>}[rr] & & M' \ar@{->>}[rr] & & M^{\gp}.\\
   U^{\cJ}\ar[rr] \ar@{ >->}[u]^{\sim} & & M \ar@{ >->}[u]_{\sim} 
  }
  \]
  Using the resulting factorization of the unit of $M^{\gp}$ into an
  acyclic cofibration $U^{\cJ}\to U(M^{\gp})$ followed by a fibration
  $U(M^{\gp})\to M^{\gp}$ for the definition of the commutative
  $\cJ$-space monoid $V(M^{\gp})$ studied in the last section, we
  obtain $V(M^{\gp})$ as an iterated pullback as indicated in the
  previous diagram. The above maps induce the following commutative
  cube:
  \[
  \xymatrix@-1.5pc{
    & B^{\rep}(M)\ar[rr]\ar'[d][dd] & & B^{\cy}(M^{\gp}) \ar[dd]\\
    M\boxtimes V(M^{\gp}) \ar[rr] \ar[dd]  \ar[ur] &  & M^{\gp}\boxtimes V(M^{\gp})\ar[dd]  \ar[ur] & \\
    & M' \ar'[r][rr]& & M^{\gp}. \\
    M\boxtimes U(M^{\gp}) \ar[rr] \ar[ur] & & M^{\gp}\boxtimes U(M^{\gp}) \ar[ur]
  }
  \]
  The back face is homotopy cartesian by definition, and the front
  face is homotopy cartesian by~\cite[Lemma
  2.11]{Sagave_log-on-k-theory} and~\cite[Corollary
  11.4]{Sagave-S_diagram}. The map in the upper right hand corner is a
  $\cJ$-equivalence by
  Proposition~\ref{prop:grouplike-M-VM-equivalence} (applied to
  $M^{\gp}$), and the maps in the lower corners are $\cJ$-equivalences
  by construction. It follows that $M \boxtimes V(M^{\gp}) \to
  B^{\rep}(M)$ is also a $\cJ$-equivalence.

  Extending the $\cJ$-space maps in the chain of $\cJ$-equivalences of
  Proposition~\ref{prop:grouplike-V(M)-equivalence} (applied to
  $M^{\gp}$) to $M$-module maps shows that there is a chain of
  $\cJ$-equivalences of $M$-modules over $M^{\gp}$ relating $M
  \boxtimes V(M^{\gp})$ and $M\times B((M^{\gp})_{h\cJ})$. Since
  $B(M_{h\cJ}) \to B((M^{\gp})_{h\cJ})$ is a weak equivalence, the
  claim follows.
\end{proof}
Let $f\colon M\to N$ be a map of commutative $\cJ$-space monoids. In
Section~\ref{sec:elog-etaleness} we shall be interested in the diagram
of commutative symmetric ring spectra
\begin{equation}\label{eq:M-BrepM-diagram}
  \xymatrix@-1pc{
    \bS^{\cJ}[M] \ar[r]\ar[d] & \bS^{\cJ}[B^{\rep}(M)]\ar[d]\\
    \bS^{\cJ}[N] \ar[r] & \bS^{\cJ}[B^{\rep}(N)]
  }
\end{equation}
induced by $f$. In order to measure to what extent this square is
homotopy cocartesian in $\cC\Spsym$, we use the following terminology: Given a
symmetric spectrum $E$, a commutative diagram of commutative symmetric
ring spectra
\[
\xymatrix@-1pc{
  A \ar[r] \ar[d] & B\ar[d]\\
  C\ar[r] & D }
\]
is \emph{$E_*$-homotopy cocartesian} if whenever
we factor $A\to C$ as a cofibration $A\to C'$ followed by a stable
equivalence $C'\to C$ of commutative symmetric ring spectra, the
induced map $C'\wedge_AB\to D$ is an $E_*$-equivalence.

\begin{proposition}\label{prop:M-N-rep-homotopy-cocartesian}
  For a symmetric spectrum $E$ and a map of cofibrant commutative
  $\cJ$-space monoids $f\colon M\to N$, the diagram
  \eqref{eq:M-BrepM-diagram} is $E_*$-homotopy cocartesian if and only if $f$ gives rise to an
  $E_*$-equivalence
  \[
  \bS^{\cJ}[N]\wedge B(M_{h\cJ})_+\to \bS^{\cJ}[N]\wedge
  B(N_{h\cJ})_+.
  \]
\end{proposition}
\begin{proof}
  Without loss of generality, we may assume that $f$ is a cofibration.
  Then the diagram~\eqref{eq:M-BrepM-diagram} is $E_*$-homotopy cocartesian
  if and only if the map 
  \begin{equation}\label{eq:M-N-rep-homotopy-cocartesian-map}
    \bS^{\cJ}[N]\sm_{\bS^{\cJ}[M]} \bS^{\cJ}[B^{\rep}(M)] \to
    \bS^{\cJ}[B^{\rep}(N)]\end{equation} is an $E_*$-equivalence. By the argument
  given in the proof of~\cite[Lemma~4.8]{RSS_LogTHH-I}, the
  extension of scalars functor $\bS^{\cJ}[N]\sm_{\bS^{\cJ}[M]}(-)$ preserves
  stable equivalences. Since the $\cJ$-equivalences of $M$-modules in
  Proposition~\ref{prop:BrepM-simeq-MtimesBMhJ} are augmented over the cofibrant
  commutative $\cJ$-space monoid $M^{\gp}$, it follows from~\cite[Corollary~8.8]{RSS_LogTHH-I} that $\bS^{\cJ}$ maps
  these $\cJ$-equivalences to stable equivalences. Hence the map~\eqref{eq:M-N-rep-homotopy-cocartesian-map} is stably equivalent to the map 
  \[\bS^{\cJ}[N]\sm_{\bS^{\cJ}[M]} (\bS^{\cJ}[M] \sm B(M_{h\cJ})_+) \to
     \bS^{\cJ}[N] \sm B(N_{h\cJ})_+,
 \] 
 and the domain of this map is isomorphic to $\bS^{\cJ}[N] \sm
 B(M_{h\cJ})_+$.\end{proof} 

\begin{notation}\label{not:J-degree-parts}
For each integer~$n$, let $\cJ_n \subset \cJ$ be the full subcategory
generated by the objects $(\bm_1, \bm_2)$ with $m_2-m_1 = n$.  Then
$B\cJ =\coprod_n B\cJ_n$ and we refer to the part of a $\cJ$-space $X$
that maps to $B\cJ_n$ as the \emph{$\cJ$-degree~$n$ part} of $X$.  If
$M$ is a commutative $\cJ$-space monoid, we let $M_{\{0\}}$ and
$B^{\rep}_{\{0\}}(M)$ denote the $\cJ$-degree $0$ parts of $M$ and
$B^{\rep}(M)$, respectively. We also use the notations
$M_{\ge0}$ and $M_{>0}$ for the non-negative and positive $\cJ$-degree
parts of $M$, respectively,
cf.~\cite[Definition~6.1]{RSS_LogTHH-I}.
\end{notation}

\begin{remark}\label{rem:repetitive} Let $M$ be a commutative $\cJ$-space monoid that is
  \emph{repetitive} in the sense
  of~\cite[Definition~6.4]{RSS_LogTHH-I}. By
  definition, this means that $M \neq M_{\{0\}}$ and that the group
  completion map $M \to M^{\gp}$ induces a $\cJ$-equivalence $M \to
  (M^{\gp})_{\geq 0}$.
  Propositions~\ref{prop:grouplike-M-VM-equivalence} and
  \ref{prop:BrepM-simeq-MtimesBMhJ} can be used to identify the
  homotopy cofiber of the map of
  $\bS^{\cJ}[B^{\cy}(M_{\{0\}})]$-module
  spectra \begin{equation}\label{ea:BcyM0-Brep0M}
    \bS^{\cJ}[B^{\cy}(M_{\{0\}})] \xrightarrow{\bS^{\cJ}[\sigma]}
    \bS^{\cJ}[B^{\rep}_{\{0\}}(M)]\end{equation} with $\Sigma\
  \bS^{\cJ}[B^{\cy}(M_{\{0\}})]$.  The idea for this is to use that
  the homotopy cofiber in question is equivalent to the homotopy
  cofiber of the map \[\bS^{\cJ}[M_{\{0\}}] \sm
  B((M_{\{0\}})_{h\cJ})_+ \to \bS^{\cJ}[M_{\{0\}}] \sm
  B(M_{h\cJ})_+.\] An application of the Bousfield--Friedlander
  theorem shows that there is a homotopy fiber
  sequence \[B((M_{\{0\}})_{h\cJ}) \to B(M_{h\cJ}) \to B(d\bN_0),\]
  where $d$ is the period of $M$, as
  in~\cite[Definition~6.5]{RSS_LogTHH-I}. Hence
  we can recognize the homotopy cofiber of~\eqref{ea:BcyM0-Brep0M}, as
  claimed.

  Using this argument in place of~\cite[Proposition~6.11]{RSS_LogTHH-I} leads to a
  slightly different proof for the localization homotopy cofiber
  sequences established in~\cite{RSS_LogTHH-I}. However, the
  disadvantage of the alternative approach outlined here is that it does
  not identify the homotopy cofiber as a cyclic object.
\end{remark}
\section{Logarithmic topological Hochschild homology}\label{sec:log-thh}
Let $A$ be a commutative symmetric ring spectrum. A \emph{pre-log
  structure} $(M,\alpha)$ on $A$ is a commutative $\cJ$-space monoid
$M$ together with a map $\alpha \colon M \to \Omega^{\cJ}(A)$;
see~\cite[Definition~4.1]{RSS_LogTHH-I}. The
ring spectrum $A$ together with a chosen pre-log structure is called a
\emph{pre-log ring spectrum} and will be denoted by $(A,M,\alpha)$ or
just $(A,M)$. (As explained
in~\cite[Remark~4.2]{RSS_LogTHH-I},
this terminology differs from the one used in~\cite{Rognes_TLS} in
that we use $\cJ$-spaces, and from~\cite[\S 4.30]{Sagave-S_diagram}
and~\cite{Sagave_log-on-k-theory} in that we skip the additional word
\emph{graded} used there.)

A basic example of a pre-log structure is the free pre-log structure
generated by a $0$-simplex $x \in
\Omega^{\cJ}(A)(\bld{d_1},\bld{d_2})$. It is given by the map
\[ \C{\bld{d_1},\bld{d_2}} = \textstyle\coprod_{k\geq 0}\left(
  F^{\cJ}_{(\bld{d_1},\bld{d_2})}(*)^{\boxtimes
    k}/\Sigma_k \right) \to \Omega^{\cJ}(A)\] from the free commutative $\cJ$-space
monoid $\C{\bld{d_1},\bld{d_2}}$ on a generator in
bidegree~$(\bld{d_1},\bld{d_2})$ determined by $x$. We often write
$\C{x}$ for $\C{\bld{d_1},\bld{d_2}}$ when discussing this map.

A more interesting kind of pre-log structure arises as follows: If
$j\colon e \to E$ is the connective cover map of a positive fibrant
commutative symmetric ring spectrum, then the pullback
$j_*\!\GLoneJof(E)$ of $\GLoneJof(E)\to \Omega^{\cJ}(E) \ot
\Omega^{\cJ}(e)$ defines a pre-log structure $j_*\!\GLoneJof(E) \to
\Omega^{\cJ}(e)$ on $e$. We call this the \emph{direct image}
pre-log structure on~$e$ induced by the trivial pre-log structure
on~$E$.

\begin{definition}\cite[Definition~4.6]{RSS_LogTHH-I}
Let $(A,M)$ be a pre-log ring spectrum. Its \emph{logarithmic
  topological Hochschild homology} $\THH(A,M)$ is the commutative
symmetric ring spectrum given by the pushout of the diagram
\[ \THH(A) \ot \bS^{\cJ}[B^{\cy}(M^{\mathrm{cof}})] \to
\bS^{\cJ}[B^{\rep}(M^{\mathrm{cof}})] \] of commutative symmetric ring
spectra. Here $(A^{\mathrm{cof}},M^{\mathrm{cof}}) \to (A,M)$ is a
cofibrant replacement and $\THH(A) = B^{\cy}(A^{\mathrm{cof}})$ is the
topological Hochschild homology of~$A$, defined as the cyclic bar
construction of $A^{\mathrm{cof}}$ with respect to the smash product
$\sm$. The left hand map is induced by the identification
$\bS^{\cJ}[B^{\cy}(M^{\mathrm{cof}})] \xrightarrow{\cong} \THH(\bS^{\cJ}[M^{\mathrm{cof}}])$ and the
adjoint structure map $\bS^{\cJ}[M^{\mathrm{cof}}] \to
A^{\mathrm{cof}}$ of $(A^{\mathrm{cof}},M^{\mathrm{cof}})$. The right
hand map is induced by the repletion map $\rho \colon
B^{\cy}(M^{\mathrm{cof}}) \to B^{\rep}(M^{\mathrm{cof}})$.
\end{definition}

When computing $\THH(A,M)$ in examples, it will be convenient to work
with pre-log structures $(M,\alpha)$ such that the homology of the
space $M_{h\cJ}$ associated with~$M$ is well understood. To obtain
interesting examples of such pre-log structures, we
review~\cite[Construction 4.2]{Sagave_log-on-k-theory}:
\begin{construction}\label{constr:D-x}
  Let $E$ be a positive fibrant commutative symmetric ring spectrum
  that is $d$-periodic, i.e., $\pi_*(E)$ has a unit of positive degree
  and the natural number $d$ is the minimal positive degree of a unit
  in $\pi_*(E)$. We also assume not to be in the degenerate case where
  $\pi_*(E)$ is the zero ring. Let $j\colon e \to E$ be the connective
  cover map and assume that $e$ is also positive fibrant.  Then there
  exists an object $(\bld{d_1},\bld{d_2})$ of $\cJ$ with $d_1 > 0$ and
  a map $x\colon S^{d_2} \to e_{d_1}$ such that $d=d_2 -d_1$ and the
  homotopy class $[x] \in \pi_{d}(e)$ represented by $x$ is mapped to
  a periodicity element in $\pi_*(E)$.

  In this general situation we will build a pre-log structure $D(x) \to
  \Omega^{\cJ}(e)$. The next diagram outlines its construction:
  \begin{equation}\label{eq:graded-direct-pre-log-from-x}
    \xymatrix@-1pc{
    \C{x} \ar@{ >->}[dr]\ar@/^.7pc/@<1ex>[ddrrrr]\ar@/^-.7pc/@<-1ex>[dddrr]\\
     & D(x) \ar@{->>}[dr]^{\sim}\\ 
      &&  D'(x)\ar[rr]\ar[d] & & \Omega^{\cJ}(e)\ar[d]\\ 
      && \C{x}^{\gp}\ar[r]& \GLoneJof(E)\ar[r]& \Omega^{\cJ}(E).
    }
  \end{equation}  
  We start with the free pre-log structure $\C{x}$ on $e$ generated by
  $x$.  The composite of its structure map $\C{x}\to \Omega^{\cJ}(e)$
  with $\Omega^{\cJ}(e) \to \Omega^{\cJ}(E)$ factors through
  $\GLoneJof(E) \to \Omega^{\cJ}(E)$ because $x$ becomes a unit in
  $\pi_*(E)$. We then factor the resulting map $\C{x} \to
  \GLoneJof(E)$ in the group completion model structure
  of~\cite{Sagave_spectra-of-units} as an acyclic cofibration $\C{x}
  \to \C{x}^{\gp}$ followed by a fibration $\C{x}^{\gp} \to
  \GLoneJof(E)$. The intermediate object $\C{x}^{\gp}$ is fibrant in
  the group completion model structure because, by construction, it
  comes with a fibration to the fibrant object $\GLoneJof(E)$. Hence the
  acyclic cofibration $\C{x} \to \C{x}^{\gp}$ is indeed a model for
  the group completion of $\C{x}$.  The commutative $\cJ$-space monoid
  $D'(x)$ is defined to be the pullback of
  \[
  \C{x}^{\gp}\to \Omega^{\cJ}(E) \ot \Omega^{\cJ}(e).
  \]
  In a final step, we define $D(x)$ by the indicated factorization of
  $\C{x} \to D'(x)$, now in the positive $\cJ$-model structure.  We
  call $D(x) \to \Omega^{\cJ}(e)$ the \emph{direct image pre-log
    structure} generated by $x$. We note that $D(x)$ is cofibrant
  since $\C{x}$ is cofibrant. Moreover, $D(x)$ is repetitive in the
  sense of Remark~\ref{rem:repetitive} since
  $\Omega^{\cJ}(e)_{\ge 0} \to \Omega^{\cJ}(E)_{\ge 0}$ is a
  $\cJ$-equivalence.
\end{construction}
It follows that there is a sequence of maps of pre-log ring spectra 
\begin{equation} \label{eq:Cx-Dx-iGloneJof-sequence}
(e,\C{x}) \to (e, D(x)) \to (e, j_*\!\GLoneJof(E)) \to (E,
\GLoneJof(E)).
\end{equation}
The significance of Construction~\ref{constr:D-x} for log $\THH$ stems
from the following results.
\begin{proposition}\label{prop:logification-of-D-of-x}
The map $(e, D(x)) \to (e, j_*\!\GLoneJof(E))$ in~\eqref{constr:D-x}
induces a stable equivalence $\THH(e, D(x)) \xrightarrow{\sim} \THH(e,
j_*\!\GLoneJof(E))$.
\end{proposition}
\begin{proof}
  By~\cite[Lemma 4.7]{Sagave_log-on-k-theory}, the map $(e, D(x)) \to
  (e, j_*\!\GLoneJof(E))$ is stably equivalent to the logification
  map. So~\cite[Theorem~4.24]{RSS_LogTHH-I} implies that it induces a
  stable equivalence when applying log $\THH$.
\end{proof}

\begin{theorem}\label{thm:localization-seq-for-Dx}
  In the situation of Construction~\ref{constr:D-x}, there is a natural
  homotopy cofiber sequence
  \begin{equation} \label{eq:cofseqedx} \THH(e)
    \overset{\rho}\longto \THH(e, D(x)) \overset{\partial}\longto \Sigma \THH(\postnikovsec{e}{0}{d})
  \end{equation}
  of $\THH(e)$-modules with circle action, where $\rho$ is a map of commutative
  symmetric ring spectra and $\postnikovsec{e}{0}{d}$ is the
  $(d-1)$-st Postnikov section of $e$. \end{theorem}
\begin{proof}
  This follows by combining \cite[Theorem~6.10 and Lemma~6.16]{RSS_LogTHH-I} or by
  combining \cite[Theorem~6.18]{RSS_LogTHH-I}
  with Proposition~\ref{prop:logification-of-D-of-x}.
\end{proof}

For later use we record a more explicit description of the homotopy
type of $D(x)$.  Since $e$ and $E$ are assumed to be positive fibrant
and $j\colon e \to E$ is the connective cover map, the induced map
$\Omega^{\cJ}(e \to E)(\bld{m_1},\bld{m_2})$ is a weak
equivalence if $m_2 - m_1 \geq 0$ and $m_1 > 0$. Moreover,
$D(x)(\bld{m_1},\bld{m_2})$ is empty if $m_2 -m_1<0$ because the
negative-dimensional units of $\pi_*(E)$ are not in the image of
the map from $\pi_*(e)$. This argument implies~\cite[Lemma 4.6]{Sagave_log-on-k-theory}
which we reproduce here for the reader's convenience:
\begin{lemma}\label{lem:hty-type-of-Dx}
  The chain of maps $\C{x}_{h\cJ} \to D(x)_{h\cJ} \to
  (\C{x}^{\gp})_{h\cJ}$ is weakly equivalent to $\coprod_{k\geq
    0}B\Sigma_k \to Q_{\geq 0}S^0 \to QS^0$. The latter chain is the
  canonical factorization of the group completion map through the
  inclusion of the non-negative components of $QS^0$.  In particular,
  $D(x) \to \C{x}^{\gp}$ induces a $\cJ$-equivalence $D(x) \to
  (\C{x}^{\gp})_{\geq 0}$.\qed
\end{lemma}
Hence the homotopy type of $D(x)_{h\cJ} \simeq Q_{\geq 0}S^0$
does not depend on the map of spectra $e \to E$ used to construct
$D(x)$. The structure map \[Q_{\geq 0}S^0 \xrightarrow{\sim} D(x)_{h\cJ} \to (\mathrm{const}_{\cJ}*)_{h\cJ} \xrightarrow{\iso} B\cJ \xrightarrow{\sim} QS^0\] is multiplication by the degree $d = d_2 -d_1$ of $x\colon S^{d_2}\to e_{d_1}$.

\section{The graded Thom isomorphism}\label{sec:graded-Thom}
As another preparatory step for computing $\THH(A,M)$, we explain how
to compute the homology of the spectrum $\bS^{\cJ}[M]$ for $M=D(x)$
and related examples.  The key idea for this, worked out by the last
two authors in~\cite{Sagave-S_graded-Thom}, is to express
$\bS^{\cJ}[M]$ as the Thom spectrum of the virtual vector bundle
classified by the composite \[M_{h\cJ} \to B\cJ \xrightarrow{\sim}
QS^{0} \to \bZ \times BO\] of the structure map $M_{h\cJ} \to B\cJ$
induced by applying $(-)_{h\cJ}$ to the map from~$M$ to the terminal
$\cJ$-space, the weak equivalence $B\cJ\to QS^{0}$, and the map of
infinite loop spaces $QS^{0} \to \bZ \times BO$ induced by the unit
$\bS \to ko$. In the case where $M = \C{\bld{d_1},\bld{d_2}}^{\gp}$
with $d_2-d_1$ even it is proved in~\cite{Sagave-S_graded-Thom} that
$M$ is orientable in the strong sense that there exists a map of
commutative symmetric ring spectra $\bS^{\cJ}[M] \to H\bZ P$. Here
$H\bZ P$ denotes a cofibrant and even periodic version of the integral
Eilenberg--Mac Lane spectrum, i.e., the underlying symmetric spectrum
of $H\bZ P$ decomposes as $H\bZ P=\bigvee_{n \in 2\bZ}H\bZ P_{\{n\}}$
where $H\bZ P_{\{0\}} = H\bZ$ and $H\bZ P_{\{n\}}=\Sigma^n H\bZ$. If
$M$ is a commutative $\cJ$-space monoid that is concentrated in even
$\cJ$-degrees, then the monoid structure of $M_{h\cJ}$ and the
multiplication of $H\bZ P$ equip $\bigvee_{n \in 2\bZ} (M_{h\cJ_{n}})_+
\sm H\bZ P_{\{n\}}$ with the structure of a symmetric ring spectrum.

In~\cite{Sagave-S_graded-Thom}, the following statement is derived
from a more general graded Thom isomorphism theorem
(in~\cite{Sagave-S_graded-Thom}, also $E_{\infty}$ structures are
addressed):
 
\begin{proposition}\cite[Proposition~8.3]{Sagave-S_graded-Thom}\label{prop:graded-thom}
  Let $(\bld{d_1},\bld{d_2})$ be an object of $\cJ$ of even
  $\cJ$-degree $d_2 - d_1$, and let $M \to
  \C{\bld{d_1},\bld{d_2}}^{\gp}$ be a map of commutative $\cJ$-space
  monoids. Then there is a chain of stable equivalences of
  symmetric ring spectra that relates $\bS^{\cJ}[M] \sm H\bZ$
  and $\bigvee_{n \in 2\bZ} (M_{h\cJ_{n}})_+ \sm H\bZ P_{\{n\}}$.  The chain
  of maps is natural with respects to maps of commutative $\cJ$-space
  monoids over $\C{\bld{d_1},\bld{d_2}}^{\gp}$.\qed
\end{proposition}
In the proposition we do not need to assume that $M$ is cofibrant
since the existence of an augmentation to the cofibrant object $\C{\bld{d_1},\bld{d_2}}^{\gp}$
ensures that $\bS^{\cJ}[M]$ captures the correct homotopy type (see~\cite[Section 8]{RSS_LogTHH-I}).

When working with the isomorphism of homology algebras resulting from
Proposition~\ref{prop:graded-thom}, it will be convenient to view the
homology of $M_{h\cJ}$ as a $\bZ$-graded algebra in a way that takes
the $\cJ$-grading into account. We use a $\circledast$ as a subscript to denote
this new grading and set
\begin{equation}\label{eq:star-grading}
H_\circledast(M_{h\cJ};\bZ) = \bigoplus_{n\in\bZ} \Sigma^n H_*(M_{h\cJ_{n}};\bZ),
\end{equation} 
and similarly for other coefficient rings. In this notation,
Proposition~\ref{prop:graded-thom} implies the following statement.
\begin{proposition}\label{prop:thom-iso-homology}
  Let $(\bld{d_1},\bld{d_2})$ be an object of $\cJ$ of even
  $\cJ$-degree $d_2 - d_1$, and let~$M$ be a commutative $\cJ$-space
  monoid over $\C{\bld{d_1},\bld{d_2}}^{\gp}$.  Then
  $H_*(\bS^\cJ[M];\bZ)$ and $H_\circledast(M_{h\cJ};\bZ)$ are
  naturally isomorphic as $\bZ$-graded algebras.\qed
\end{proposition}
If $x\colon S^{d_2} \to e_{d_1}$ has even degree $d = d_2-d_1$, then
the two previous propositions apply for example to the commutative
$\cJ$-space monoids $D(x)$, $D(x)^{\gp}$, $B^{\cy}(D(x))$,
$B^{\rep}(D(x))$, $B^{\cy}(D(x)_{\{0\}})$, and
$B^{\cy}_{\{0\}}(D(x)^{\gp})$. Here  $B^{\cy}(D(x)_{\{0\}})$ and
$B^{\cy}_{\{0\}}(D(x)^{\gp})$ denote the $\cJ$-degree zero parts; see Notation~\ref{not:J-degree-parts}. In view of later applications, we
formulate the following results for homology with
$\bF_p$-coefficients.
\begin{corollary} Let $x$ have even degree $d$. There are algebra isomorphisms 
\begin{align*}
H_*(\bS^{\cJ}[D(x)];\bF_p) \iso &\, H_\circledast( D(x)_{h\cJ};\bF_p) \cong P(x) \otimes H_*(D(x)_{h\cJ_{0}} ;\bF_p)\\
H_*(\bS^{\cJ}[D(x)^{\gp}];\bF_p) \iso &\, H_\circledast( D(x)^{\gp}_{h\cJ};\bF_p) \cong P(x^{\pm1}) \otimes H_*(D(x)_{h\cJ_{0}};\bF_p)
\end{align*}
 with
$H_*(D(x)_{h\cJ_{0}};\bF_p) \cong H_*(Q_0S^0;\bF_p)$ in $\cJ$-degree~$0$.
\end{corollary}
\begin{proof}
  The first isomorphisms follow from the above observations. For a
  general simplicial monoid $A$ that is grouplike and homotopy
  commutative, $A$ and $\pi_0(A)\times A_0$ are weakly equivalent as
  $H$-spaces. Here $A_0$ denotes the connected component of the unit
  element in~$A$. This provides the second isomorphisms. The last
  statement follows from Lemma~\ref{lem:hty-type-of-Dx}.
\end{proof}

\subsection{Homology of the cyclic and replete bar constructions}
To describe the homology of $\bS^{\cJ}[B^{\cy}(D(x))] $ and
$\bS^{\cJ}[B^{\rep}(D(x))] $, we write
\begin{equation}\label{eq:def-C-HBcyDx0}
C_* = H_* (B^{\cy}(D(x)_{\{0\}})_{h\cJ};\bF_p) 
\end{equation}
for the homology algebra of the $E_{\infty}$ space
$B^{\cy}(D(x)_{\{0\}})_{h\cJ}$. If $k$ is a positive integer, we say
that a $\bZ$-graded $\bF_p$-algebra $A$ is $k$-connected if $\bF_p\cong A_0$ and $A_i = 0$ if $i<0$ or $0<i\leq k$. In view of
Corollary~\ref{cor:BcyM-hJ-vs-Bcy-of-MhJ}, the underlying $\bZ$-graded
$\bF_p$-vector space of $C_*$ can be identified with
$H_*(B^{\cy}(D(x)_{h\cJ_{0}}); \bF_p)$ when $x$ has non-zero
degree. Since we have not defined a multiplicative structure on
$B^{\cy}(D(x)_{h\cJ_{0}})$, we cannot view this vector space
isomorphism as an algebra isomorphism. Nonetheless, the isomorphism
implies that $C_*$ is $(2p-4)$-connected, since
$B^{\cy}(D(x)_{h\cJ_{0}})$ is weakly equivalent to $D(x)_{h\cJ_{0}}
\times B(D(x)_{h\cJ_{0}})$ by the argument given in the proof of
Lemma~\ref{lem:Bcy-bar-M-decomposition}, and
$H_*(D(x)_{h\cJ_{0}};\bF_p)$ is $(2p-4)$-connected since
$D(x)_{h\cJ_{0}} \simeq Q_0S^0$ by Lemma~\ref{lem:hty-type-of-Dx}.
\begin{proposition}\label{prop:homology-B^cyDx}
  Let $p \geq 3$ be an odd prime and assume that $x$ has positive even degree. There are algebra
  isomorphisms
\begin{align*}
  H_*(\bS^{\cJ}[B^{\cy}(D(x))];\bF_p) \iso & \, H_\circledast(B^{\cy}(D(x))_{h\cJ};\bF_p) \cong P(x) \otimes E(dx) \otimes C_* \\
  H_*(\bS^{\cJ}[B^{\rep}(D(x))];\bF_p) \iso & \, H_\circledast(B^{\rep}(D(x))_{h\cJ} ;\bF_p)\cong P(x) \otimes E(d\log x) \otimes C_* \\
  H_*(\bS^{\cJ}[B^{\cy}_{\{0\}}(D(x)^{\gp})];\bF_p) \iso & \, H_\circledast(B^{\cy}(D(x)^{\gp})_{h\cJ_{0}};\bF_p)\cong E(d\log x) \otimes C_*
\end{align*}
with $|dx|= |x| + 1$, $|d\log x|= 1$, and $dx$ mapping to $x
\cdot d\log x$ under the repletion map. The suspension operator
satisfies $\sigma(x) = dx$, $\sigma(dx) = 0$ in the first case, and $\sigma(x) = x
\cdot d\log x$, $\sigma(d\log x) = 0$ in the second case.
\end{proposition}
We need some preparation to prove the proposition. First we recall
from~\cite[Section~7]{RSS_LogTHH-I} that for a commutative $\cJ$-space
monoid $M$ concentrated in $\cJ$-degrees divisible by $d$, there is a
natural augmentation map $B^{\cy}(M)_{h\cJ} \to B^{\cy}(d\bZ)$ which
is defined as the realization of the map
\[
B^{\cy}_s(M)_{h\cJ} = ( \textstyle\coprod_{(d_0,\dots, d_s)\in B^{\cy}_s(d\bZ)} M_{\{d_0\}} \boxtimes \dots \boxtimes M_{\{d_s\}})_{h\cJ} \to B^{\cy}_s (d\bZ)  
\]
that collapses the summand indexed by $(d_0,\dots, d_s)$ to the point
$(d_0,\dots, d_s)$. 

The category of simplicial monoids has a model structure in which a map is a fibration 
or weak equivalence if and only if the underlying map of simplicial sets is. 
Specializing to the case $M=D(x)^{\gp}$, we choose a factorization 
\[\xymatrix{B^{\cy}(D(x)^{\gp})_{h\cJ} \ar@{ >->}[r]^{\sim} & B^{\cy}(D(x)^{\gp})^{\mathrm{fib}}_{h\cJ} \ar@{->>}[r]^-{q}& B^{\cy}(d\bZ)}\]
of the augmentation in this model structure. 
\begin{lemma}\label{lem:section-to-BDxgp-BdZ}
There is a basepoint preserving space level section to $q$.  
\end{lemma}
\begin{proof}
  Let $d$ be the degree of $x$. Since $D(x)^{\gp}_{h\cJ} \simeq QS^0$,
  the canonical augmentation $D(x)^{\gp}_{h\cJ} \times
  B(D(x)^{\gp}_{h\cJ}) \to d\bZ \times B(d\bZ)$ admits a section in
  $\mathrm{Ho}(\cS_*)$, the homotopy category of based simplicial
  sets. The chain of weak equivalences between
  $B^{\cy}(D(x)^{\gp})_{h\cJ}$ and $D(x)^{\gp}_{h\cJ} \times
  B(D(x)^{\gp}_{h\cJ})$ resulting from
  Corollary~\ref{cor:BcyM-hJ-vs-Bcy-of-MhJ} and the proof of
  Lemma~\ref{lem:Bcy-bar-M-decomposition} is basepoint preserving and
  compatible with the augmentation to $B^{\cy}(d\bZ) \simeq d\bZ
  \times B(d\bZ)$. Hence the augmentation $B^{\cy}(D(x)^{\gp})_{h\cJ}
  \to B^{\cy}(d\bZ)$ also admits a section in $\mathrm{Ho}(\cS_*)$. 

  It follows that the map~$q$ is a fibration of cofibrant and fibrant
  based simplicial sets, which admits a section in the homotopy
  category.  By the homotopy lifting property it therefore admits a
  section in the category $\mathcal{S}_*$ of based simplicial sets.
\end{proof}

\begin{proof}[Proof of Proposition~\ref{prop:homology-B^cyDx}]
The graded Thom isomorphism  provides the
first isomorphism in each line of the statement. For the second, we recall
from~\cite{Rognes_TLS}*{Propositions~3.20 and 3.21}
or~\cite[Section~5.2]{RSS_LogTHH-I} that 
there are algebra isomorphisms
\begin{align*}
H_*(B^{\cy}(\bN_0);\bF_p)\cong & \,  H_*(*\sqcup \textstyle\coprod_{k\geq 1} S^1(k); \bF_p) \cong P(x) \otimes E(dx),\\
H_*(B^{\rep}(\bN_0);\bF_p) \cong & \,  H_*(\textstyle\coprod_{k\geq 0} S^1(k); \bF_p) \cong  P(x) \otimes E(d\log x)\, ,\text{ and}\\
H_*(B^{\cy}_{\{0\}}(\bZ);\bF_p) \cong  & \, H_*(S^1(0); \bF_p) \cong   E(d\log x). 
\end{align*}
Here each $S^1(k)$ is a $1$-sphere, and we have $x \in H_0(S^1(1);\bF_p)$, $dx \in
H_1(S^1(1);\bF_p)$, and $d\log x \in H_1(S^1(0);\bF_p)$. 

We first treat the case of $H_*(B^{\cy}(D(x))_{h\cJ};\bF_p)$. Writing $d$ for the degree of $x$, 
we observe that the augmentations induce a commutative diagram
\begin{equation}\label{eq:homology-B^cyDx-cart-squares}
\xymatrix@-1pc{
B^{\cy}(D(x)_{\{0\}})_{h\cJ} \ar[d] \ar[rr] && B^{\cy}(D(x))_{h\cJ}  \ar[d] \ar[rr] && B^{\cy}(D(x)^{\gp})_{h\cJ} \ar[d] \\
{*} \ar[rr] && B^{\cy}(d\bN_0) \ar[rr]&& B^{\cy}(d\bZ)\,.
}
\end{equation}
Using the weak equivalences from
Corollary~\ref{cor:BcyM-hJ-vs-Bcy-of-MhJ} and applying the
Bousfield--Fried\-lander theorem as in the
proof~\cite[Proposition~7.1]{RSS_LogTHH-I} shows
that that the outer rectangle and the right hand square in this
diagram are homotopy cartesian. Hence so is the left hand square. 

Let $\pi\colon
\xymatrix@1{B^{\cy}(D(x))^{\mathrm{fib}}_{h\cJ}\ar@{->>}[r]&
  B^{\cy}(d\bN_0)}$ be the fibration of simplicial monoids obtained by
forming the base change of the fibration $q$ considered in
Lemma~\ref{lem:section-to-BDxgp-BdZ} 
along $B^{\cy}(d\bN_0) \to B^{\cy}(d\bZ)$. Then the canonical map
$\iota \colon B^{\cy}(D(x))_{h\cJ}\to
B^{\cy}(D(x))^{\mathrm{fib}}_{h\cJ}$ is a weak equivalence since the
right hand square in~\eqref{eq:homology-B^cyDx-cart-squares} is
homotopy cartesian, and it follows from
Lemma~\ref{lem:section-to-BDxgp-BdZ} that $\pi$ admits a basepoint
preserving space level section~$\tau$.

For a $0$-simplex $c \in B^{\cy}(d\bN_0)$, we now consider the diagram 
\begin{equation}\label{eq:homology-B^cyDx-mult-we}\xymatrix@-1pc{
B^{\cy}(D(x)_{\{0\}})_{h\cJ} \ar[rr]^-{\mathrm{incl}_c} \ar@{=}[d] && B^{\cy}(D(x)_{\{0\}})_{h\cJ} \times B^{\cy} (d\bN_0)  \ar[rr]^-{\mathrm{proj}} \ar[d]^{\mu} && B^{\cy} (d\bN_0) \ar@{=}[d] \\
B^{\cy}(D(x)_{\{0\}})_{h\cJ} \ar[rr]^{\nu_c} &&  B^{\cy}(D(x))_{h\cJ}^{\mathrm{fib}} \ar[rr]^{\pi} && B^{\cy} (d\bN_0)\, .
}
\end{equation}
where $\mathrm{incl}_c(a) = (a,c)$, $\mu(a,b) = \iota(a)\cdot \tau(b)$, and
$\nu_c(a) = \iota(a)\cdot \tau(c)$. By
construction, both squares commute. We want to show that $\mu$ is a
weak equivalence.  Since the left hand square
in~\eqref{eq:homology-B^cyDx-cart-squares} is homotopy cartesian, the
lower sequence in~\eqref{eq:homology-B^cyDx-mult-we} is a homotopy
fiber sequence if $c$ is the basepoint. A five-lemma argument with the
long exact sequences and a consideration of path components shows that
$\mu$ induces an isomorphism on homotopy groups with basepoints in the
zero component. Since in the
diagram~\eqref{eq:homology-B^cyDx-cart-squares}, both squares are
homotopy cartesian, the rightmost map is a homomorphism of grouplike
simplicial monoids, and $B^{\cy}_{>0}(d\bN_0) \to B^{\cy}_{>0}(d\bZ)$ is
a weak equivalence, it
follows that the lower sequence in~\eqref{eq:homology-B^cyDx-mult-we}
is also a homotopy fiber sequence if $c$ lies in a positive path
component. Arguing again with the long exact sequence completes the
proof of $\mu$ being a weak equivalence.

With the notation $C_* = H_*(B^{\cy}(D(x)_{\{0\}})_{h\cJ};\bF_p)$ from~\eqref{eq:def-C-HBcyDx0}, the K{\"u}nneth isomorphism and the weak equivalence $\mu$ induce an isomorphism \[
C_* \tensor H_*(B^{\cy}(d\bN_0);\bF_p) \iso H_*(B^{\cy}(D(x))_{h\cJ};\bF_p).
\]
It is an isomorphism of $C_*$-modules under $C_*$ since $\mu$ is a map of $B^{\cy}(D(x)_{\{0\}})_{h\cJ}$-modules under $B^{\cy}(D(x)_{\{0\}})_{h\cJ}$.
We need to verify that this isomorphism
is multiplicative, and it suffices to check this on each tensor factor
on the left.  For $C_*$, this holds by construction.  The homomorphism
\[\tau_*\colon H_*(B^{\cy}(d\bN_0);\bF_p) \to
H_*(B^{\cy}(D(x))_{h\cJ};\bF_p)\] is induced by the space level section
$\tau$. Hence it is an additive section to the algebra homomorphism
$\pi_*\colon H_*(B^{\cy}(D(x))_{h\cJ};\bF_p) \to
H_*(B^{\cy}(d\bN_0);\bF_p)$. The algebra $C_*$ is $(2p-4)$-connected
and therefore at least $2$-connected for $p\geq 3$. Hence $\pi_*$ is
an isomorphism in degrees $* \leq 2$.  Since $
H_*(B^{\cy}(d\bN_0);\bF_p)$ is concentrated in degrees $* \leq 1$,
this implies, as an algebraic fact, that the additive section $\tau_*$
is multiplicative. It also implies that the suspension operator
satisfies $\sigma(x)=dx$ and $\sigma(dx)=0$, since these relations hold
in $ H_*(B^{\cy}(d\bN_0);\bF_p)$ and  are preserved by
$\pi_*$.

The claims about $H_*(B^{\rep}(D(x))_{h\cJ};\bF_p)$ and
$H_*(B^{\cy}_{\{0\}}(D(x)^{\gp})_{h\cJ};\bF_p)$ follow by a similar argument with
$B^{\cy}_{\geq 0}(d\bZ)$ (respectively
$B^{\cy}_{\{0\}}(d\bZ)$) replacing $B^{\cy}(d\bN_0)$.
\end{proof}
We do not know if the statement of
Proposition~\ref{prop:homology-B^cyDx} holds for $p=2$. Since we are
interested in $V(1)$-homotopy calculations later on, we shall not
pursue this question further.

\subsection{Homology of the group completion of free commutative
  \texorpdfstring{$\cJ$}{J}-space monoids}
We give another application of the graded Thom isomorphism that will
become relevant for the proof of
Theorem~\ref{thm:ell-ku-direct-image-logthh-etale} below.  Let $(\bld
e_1,\bld e_2)$ be an object with $e_2-e_1 > 0$ and let $(\bld d_1,\bld
d_2)=(\bld e_1,\bld e_2) ^{\sqcup k}$ for a positive integer~$k$. We
have a map of commutative $\cJ$-space monoids $\C{\bld d_1,\bld
  d_2}\to \C{\bld e_1,\bld e_2}$ defined by mapping the generator
$\id_{(\bld d_1,\bld d_2)}$ to $\id_{(\bld e_1,\bld e_2)}^{\boxtimes
  k}$.  Consider the commutative diagram of commutative $\cJ$-space
monoids
\[
\xymatrix@-1pc{ \C{\bld d_1,\bld d_2} \ar[r]\ar[d]&
  \C{\bld d_1,\bld d_2}^{\gp}_{\geq 0} \ar[r] \ar[d]
  & \C{ \bld d_1,\bld d_2}^{\gp} \ar[d]\\
  \C{ \bld e_1,\bld e_2}\ar[r]& \C{ \bld e_1,\bld
  e_2}^{\gp}_{\geq 0} \ar[r]& \C{ \bld e_1,\bld
  e_2}^{\gp} }
\]
in which $\C{ \bld d_1,\bld d_2}^{\gp}$ and $\C{
\bld e_1,\bld e_2}^{\gp}$ are cofibrant, the horizontal
composites are group completions and the right hand horizontal maps
are the canonical inclusions.

\begin{lemma}\label{lem:free-p-local-homology}
  With notation as above set $e=e_2-e_1$, assume that $e$ is even, and
  let $p$ be a prime not dividing $k$. Then $H_*(\bS^{\cJ}[\C{ \bld
    e_1,\bld e_2}^{\gp}_{\geq 0}];\bZ_{(p)})$ is a free module over
  $H_*(\bS^{\cJ}[\C{ \bld d_1,\bld d_2}^{\gp}_{\geq 0}];\bZ_{(p)})$
  generated by the images of the canonical generators of
  $\Sigma^{ie}\bZ_{(p)}$ under the homomorphisms
  \[
  \Sigma^{ie}\bZ_{(p)}\cong H_*(\bS^{\cJ}[F^{\cJ}_{(\bld
    e_1,\bld e_2)^{\sqcup i}}(*)];\bZ_{(p)}) \to
  H_*(\bS^{\cJ}[\C{ \bld e_1,\bld
  e_2}^{\gp}_{\geq 0}];\bZ_{(p)})
  \]
  for $i=0,\dots,k-1$.
\end{lemma}
Here the last map is induced by the canonical map $F^{\cJ}_{(\bld e_1,\bld
  e_2)^{\sqcup i}}(*)\cong F^{\cJ}_{(\bld e_1,\bld e_2)}(*)^{\boxtimes i}\to \C{ \bld e_1,\bld e_2}$.
\begin{proof}
  Consider the map of $\cJ$-spaces
  \[
  F_{(\bld e_1,\bld e_2)^{\sqcup i}}^{\cJ}(*)\boxtimes \C{ \bld
  d_1,\bld d_2}^{\gp}_{\geq 0} \to \C{ \bld e_1,\bld
  e_2}^{\gp}_{\geq 0} \boxtimes \C{ \bld e_1,\bld e_2}^{\gp}_{\geq 0}
  \to \C{ \bld e_1,\bld e_2}^{\gp}_{\geq 0}
  \]
  for $i=0,\dots,k-1$. The statement in the theorem is equivalent to
  the induced map of symmetric spectra
  \[
  \bS^{\cJ}\left[ \textstyle\coprod_{i=0}^{k-1}F_{(\bld e_1,\bld e_2)^{\sqcup
      i}}^{\cJ}(*)\boxtimes \C{ \bld d_1,\bld
  d_2}^{\gp}_{\geq 0}\right]\to \bS^{\cJ}[\C{ \bld e_1,\bld
  e_2}^{\gp}_{\geq 0}]
  \]
  inducing an isomorphism on homology with
  $\bZ_{(p)}$-coefficients. By Proposition~\ref{prop:thom-iso-homology} this is equivalent to the map of
  spaces
  \[
  \textstyle\coprod_{i=0}^{k-1}\big(F_{(\bld e_1,\bld e_2)^{\sqcup
      i}}^{\cJ}(*)\boxtimes \C{ \bld d_1,\bld
  d_2}^{\gp}_{\geq 0}\big)_{h\cJ_{n}}\to \left(\C{
    \bld e_1,\bld e_2}^{\gp}_{\geq 0}\right)_{h\cJ_{n}}
  \]
  inducing an isomorphism on homology with $\bZ_{(p)}$-coefficients
  for all $n$.  We can further reduce this to the case $n=0$ by
  considering the commutative diagram
  \[
  \xymatrix@-1pc{ \big(F_{(\bld e_1,\bld e_2)^{\sqcup
        i}}^{\cJ}(*)\boxtimes \C{\bld d_1,\bld d_2}^{\gp}
    \big)_{h\cJ_{ie+jd}} \ar[r] &  \C{ \bld e_1,\bld e_2}^{\gp}_{h\cJ_{ie+jd}}\\
    F_{(\bld e_1,\bld e_2)^{\sqcup i}}^{\cJ}(*)_{h\cJ_{ie}} \times
    \C{\bld d_1,\bld d_2 }^{\gp}_{h\cJ_{jd}}\ar[r]
    \ar[u]^{\sim}& F_{(\bld e_1,\bld e_2)^{\sqcup
        i}}^{\cJ}(*)_{h\cJ_{ie}} \times \C{\bld e_1,\bld
    e_2
    }^{\gp}_{h\cJ_{jd}}\ar[u]^{\sim}\\
    F_{(\bld e_1,\bld e_2)^{\sqcup i}}^{\cJ}(*)_{h\cJ_{ie}} \times
    \C{\bld d_1,\bld d_2 }^{\gp}_{h\cJ_{0}}\ar[r]
    \ar[u]^{\sim}_{\id\times \id_{(\bld d_1,\bld d_2)^{\sqcup j}}} &
    F_{(\bld e_1,\bld e_2)^{\sqcup i}}^{\cJ}(*)_{h\cJ_{ie}} \times
    \C{\bld e_1,\bld e_2
    }^{\gp}_{h\cJ_{0}}\ar[u]^{\sim}_{\id\times \id_{(\bld
        e_1,\bld e_2)^{\sqcup jk}}} }
  \]
  where $d=d_2-d_1=ke$, and the arrows labeled $\id_{(\bld d_1,\bld d_2)^{\sqcup j}}$ and
  $\id_{(\bld e_1,\bld e_2)^{\sqcup jk}}$ are given by left
  translation by the images of these elements in the group
  completions. Since the vertical maps are weak homotopy equivalences
  as indicated, it remains to show that the map $\C{ \bld d_1,\bld
    d_2}^{\gp}_{h\cJ_{0}} \to \C{ \bld e_1,\bld
    e_2}^{\gp}_{h\cJ_{0}}$ induces an isomorphism in homology with
  $\bZ_{(p)}$-coefficients. 

  Using that $\C{ \bld{n_1},\bld{n_2}}_{h\cJ}^{\gp} \to \Omega B(\C{
    \bld{n_1},\bld{n_2}}_{h\cJ}^{\gp})$ is a weak equivalence, the
  next lemma shows that $\C{ \bld d_1,\bld d_2}^{\gp}_{h\cJ_{0}} \to
  \C{ \bld e_1,\bld e_2}^{\gp}_{h\cJ_{0}}$ induces multiplication
  by $k$ on the homotopy groups. Hence the map of homotopy groups
  becomes an isomorphism after tensoring with $\bZ_{(p)}$. This in
  turn implies that it induces an isomorphism in homology with
  $\bZ_{(p)}$-coefficients (see e.g.~\cite[Proposition V.3.2]{Bousfield-K_homotopy-limits}). 
\end{proof}

\begin{lemma}\label{lem:gp-compl-map-BJ-map}
  Let $k$ be a positive integer, let $(\bld{e_1},\bld{e_2})$ be an
  object of $\cJ$, let $(\bld{d_1},\bld{d_2})=
  (\bld{e_1},\bld{e_2})^{\concat k}$ and let $\C{\bld d_1,\bld d_2}\to
  \C{\bld e_1,\bld e_2}$ be the map determined by $\id_{(\bld d_1,\bld
    d_2)} \mapsto \id_{(\bld e_1,\bld e_2)}^{\boxtimes k}$. Then the
  induced map $ B(\C{ \bld d_1,\bld d_2}_{h\cJ}^{\gp})\to B(\C{ \bld
    e_1,\bld e_2 }_{h\cJ}^{\gp})$ acts as multiplication by $k$ on the
  homotopy groups. \end{lemma}

\begin{proof}
  Let $\bld{k}$ be the finite set $\bld{k}=\{1,\dots,k\}$ and let
  $\bld k^{\sqcup}\colon \Sigma\to\Sigma$ be the functor taking $\bld
  n$ to the $n$-fold concatenation $\bld k^{\sqcup n}$ with its
  canonical left $\Sigma_n$-action. This is a strict symmetric
  monoidal functor and we use the same notation for the induced
  functor $\bld k^{\sqcup}\colon\cJ\to\cJ$. It is easy to see that
  $\bld k^{\sqcup}\colon\cJ\to\cJ$ is naturally isomorphic to the
  $k$-fold monoidal product $(\bld n_1,\bld n_2)\mapsto (\bld n_1,\bld
  n_2)^{\sqcup k}$. Hence the induced map $B(B\cJ) \to B(B\cJ)$ acts
  as multiplication by $k$ on the homotopy groups.

  Let $\tilde\Sigma_n$ be the translation category of $\Sigma_n$. Its
  objects are the elements in $\Sigma_n$ and its morphisms $a\colon
  b\to c$ are elements in $\Sigma_n$ such that $ab=c$. We consider the
  functor $\tilde\Sigma_n\to ((\bld d_1,\bld d_2)^{\sqcup
    n}\downarrow\cJ)$ that maps an object $a$ in $\tilde \Sigma_n$ to
  the isomorphism $a_*\colon (\bld d_1,\bld d_2)^{\sqcup n}\to (\bld
  d_1,\bld d_2)^{\sqcup n}$ induced by $a$ and the symmetry
  isomorphism. This defines a weak homotopy equivalence of
  $\Sigma_n$-orbit spaces
  \[
  B\Sigma_n \cong
  (B\tilde\Sigma_n)/\Sigma_n\xrightarrow{\sim}B((\bld d_1,\bld
  d_2)^{\sqcup n}\downarrow\cJ)/\Sigma_n \iso (F_{(\bld d_1,\bld
    d_2)}^{\cJ}(*)^{\boxtimes n}/\Sigma_n)_{h\cJ}.
  \]
  Assembling these maps for varying $n$ and forming the corresponding
  map for $\C{ \bld e_1,\bld e_2}_{h\cJ}$ defines the
  horizontal weak equivalences in the following commutative diagram of
  $E_{\infty}$ monoids:
  \[
  \xymatrix@-1pc{\C{ \bld d_1,\bld d_2}_{h\cJ} \ar[d] &&& B\Sigma
    \ar[d]^{\bld k^{\sqcup}} \ar[lll]_-{(\bld
      d_1,\bld d_2)^{\sqcup}}^{\sim}\\
    \C{\bld e_1,\bld e_2}_{h\cJ} &&& \ar[lll]_-{(\bld e_1,\bld
      e_2)^{\sqcup}}^{\sim} B\Sigma } \] Since the maps $B(B\Sigma) \to
  B(B\cJ)$ and $B(\C{ \bld d_1,\bld d_2}_{h\cJ}) \to B(\C{ \bld d_1,\bld
    d_2}_{h\cJ}^{\gp})$ are weak equivalences, the claim follows by the
  above observation about the map $B(B\cJ) \to B(B\cJ)$ induced by
  $\bld k^{\sqcup}\colon\cJ\to\cJ$.
\end{proof}

\section{The inclusion of the Adams summand}\label{sec:elog-etaleness}
Let $p$ be an odd prime. It is well-known that there are positive
fibrant commutative symmetric ring spectra $ku_{(p)}$ and $\ell$
modeling respectively the $p$-local connective complex $K$-theory
spectrum and the corresponding Adams summand. Furthermore, we may
assume that the inclusion of the Adams summand $f\colon \ell\to
ku_{(p)}$ is realized as a positive fibration of commutative symmetric
ring spectra, that the corresponding periodic theories are also
represented by positive fibrant commutative symmetric ring spectra
$KU_{(p)}$ and $L$, and that there is a commutative square
\begin{equation}\label{eq:l-ku-L-KU}
\xymatrix@-1pc{ \ell \ar[d]_f \ar[r] & L \ar[d] \\ ku_{(p)} \ar[r] & KU_{(p)} }
\end{equation}
in $\cC\Spsym$. This is explained in \cite[Section
4.12]{Sagave_log-on-k-theory} and is based on work by Baker and
Richter~\cite{Baker-R_numerical}.

The graded rings of homotopy groups are given by
$\pi_*(ku_{(p)})=\bZ_{(p)}[u]$ with $|u|=2$ and
$\pi_*(\ell)=\bZ_{(p)}[v]$ with $|v|=2(p-1)$, and under these
isomorphisms $f$ corresponds to the ring homomorphism $\bZ_{(p)}[v]\to
\bZ_{(p)}[u]$ taking $v$ to $u^{p-1}$. It is proved in
\cite[Proposition 4.13]{Sagave_log-on-k-theory} that one may choose
representatives $v$ in $\Omega^{\cJ}(\ell)(\bld{p-1},\bld{3(p-1)})$
and $u$ in $\Omega^{\cJ}( ku_{(p)})(\bld 1,\bld 3)$ as well as a map
$f^{\flat}\colon D(v) \to D(u)$ relating the commutative $\cJ$-space
monoids $D(v)$ and $D(u)$ from Construction~\ref{constr:D-x} such that
there is a commutative diagram of pre-log ring spectra
\begin{equation}\label{eq:l-ku-Dv-Du-pre-log}\xymatrix@-1pc{
(\ell, D(v)) \ar[d] \ar[rr] && (\ell, j_*\!\GLoneJof(L)) \ar[d] \\
(ku_{(p)}, D(u)) \ar[rr] && (ku_{(p)}, j_*\!\GLoneJof( KU_{(p)})).  
}
\end{equation}
In the diagram, the horizontal maps are induced by the maps $D(v) \to
j_*\!\GLoneJof(L)$ and $D(u) \to j_*\!\GLoneJof(KU_{(p)})$
from~\eqref{eq:Cx-Dx-iGloneJof-sequence}.

\subsection{Formally log \texorpdfstring{$\THH$}{THH}-\'{e}tale maps}
In analogy with the notion of \emph{formally $\THH$-{\'e}tale} maps
defined in~\cite[\S 9.2]{Rognes_Galois}, we say that a map of pre-log ring spectra
 $(A,M) \to (B,N)$
is \emph{formally log $\THH$-{\'e}tale} if the induced square
\[\xymatrix@-1pc{A \ar[rr] \ar[d] && \THH(A,M) \ar[d] \\ B \ar[rr] && \THH(B,N)}\]
is a homotopy cocartesian diagram of commutative symmetric ring spectra. 
\begin{theorem}\label{thm:ell-ku-direct-image-logthh-etale}
  The map of log ring spectra $(\ell, j_*\!\GLoneJof(L)) \to
  (ku_{(p)}, j_*\!\GLoneJof(KU_{(p)}))$ is formally \emph{log
    $\THH$-{\'e}tale}.
\end{theorem}
By Proposition~\ref{prop:logification-of-D-of-x}, the horizontal maps
in~\eqref{eq:l-ku-Dv-Du-pre-log} induce stable equivalences when
applying $\THH$. This reduces
Theorem~\ref{thm:ell-ku-direct-image-logthh-etale} to the next
statement.
\begin{theorem}\label{thm:ell-ku-Du-Dv-logthh-etale}
  The map of pre-log ring spectra $(\ell, D(v)) \to
  (ku_{(p)}, D(u))$ is formally \emph{log
    $\THH$-{\'e}tale}.
\end{theorem}

\begin{proof}
We need to show that the square given by the back face of the following
commutative diagram in $\cC\Spsym$ is homotopy cocartesian:
\[\xymatrix@-1.5pc{
  & \ell \ar'[d][dd] \ar[rr]& & ku_{(p)}
  \ar[dd]\\ \bS^{\cJ}[D(v)]\ar[ur]\ar[dd] \ar[rr] & &
  \bS^{\cJ}[D(u)]\ar[dd]\ar[ur] \\ & \THH(\ell) \ar'[d][dd]
  \ar'[r][rr] & & \THH(ku_{(p)}) \ar[dd]
  \\ \bS^{\cJ}[B^{\cy}(D(v))]\ar[dd] \ar[ur]\ar[rr] & &
  \bS^{\cJ}[B^{\cy}(D(u))] \ar[ur]\ar[dd]\\ & \THH(\ell, D(v))
  \ar'[r][rr]& & \THH(ku_{(p)}, D(u))\, .\\ \bS^{\cJ}[B^{\rep}(D(v))]
  \ar[ur]\ar[rr] & & \bS^{\cJ}[B^{\rep}(D(u))] \ar[ur] }\] The lower
left hand face and the lower right hand face are homotopy cocartesian
squares by definition. The inner square is homotopy cocartesian since
it results from applying $\THH$ to the homotopy cocartesian square of
Proposition~\ref{prop:Dv-Du-ell-ku-hty-cocartesian} below. Hence the bottom
face of the diagram is homotopy cocartesian. The top face of the
diagram is homotopy cocartesian by
Proposition~\ref{prop:Dv-Du-ell-ku-hty-cocartesian}. Since the back face is
already $p$-local, it is therefore enough to show that the front face
becomes a homotopy cocartesian square after $p$-localization.

Using Proposition~\ref{prop:M-N-rep-homotopy-cocartesian} it suffices
to show that $B(D(v)_{h\cJ})\to B(D(u)_{h\cJ})$ induces an isomorphism
in homology with $\bZ_{(p)}$-coefficients. By
Lemma~\ref{lem:hty-type-of-Dx}, this is equivalent to showing that $
B(\C{ v}^{\gp}_{h\cJ})\to B(\C{u}^{\gp}_{h\cJ})$ induces an isomorphism
in homology with $\bZ_{(p)}$-coefficients. Lemma~\ref{lem:gp-compl-map-BJ-map} shows that the latter
map acts as multiplication by $(p-1)$ on the homotopy groups. Hence it
induces an isomorphism on homotopy groups after tensoring with
$\bZ_{(p)}$, which in turn implies that the map in homology with
$\bZ_{(p)}$-coefficients is an isomorphism.
\end{proof}

The next lemma was used in the proof of
Theorem~\ref{thm:ell-ku-Du-Dv-logthh-etale}. The lemma is identical
to~\cite[Proposition 4.15]{Sagave_log-on-k-theory}, but we provide a
more conceptual argument that is based on the graded Thom isomorphism
of Proposition~\ref{prop:graded-thom}:
\begin{proposition}\label{prop:Dv-Du-ell-ku-hty-cocartesian}
The commutative square of commutative symmetric ring spectra
\[\xymatrix@-1pc{
\bS^{\cJ}[D(v)] \ar[r] \ar[d] &  \ell \ar[d]\\
 \bS^{\cJ}[D(u)] \ar[r]  & ku_{(p)}
}\]
is homotopy cocartesian. 
\end{proposition}
\begin{proof}
  If we factor $\bS^{\cJ}[D(v)]\to \ell$ as a cofibration
  $\bS^{\cJ}[D(v)]\to \ell'$ followed by an acyclic fibration
  $\ell'\to \ell$, then we have to show that the induced map
  \[
  \bS^{\cJ}[D(u)]\wedge_{\bS^{\cJ}[D(v)]}\ell'\to ku_{(p)}
  \]
  is a stable equivalence. These are $p$-local connective spectra, so
  it suffices to show that the latter map induces an isomorphism in
  spectrum homology with $\bZ_{(p)}$-coefficients.  Consider for each
  $i=0,\dots,p-2$ the map of symmetric spectra
  \[
  \bS^{\cJ}[F^{\cJ}_{(\bld 1,\bld 3)^{\sqcup i}}(*)]\to
  \bS^{\cJ}[\C{u}]\to ku_{(p)}.
  \]
  The symmetric spectrum $\bS^{\cJ}[F^{\cJ}_{(\bld 1,\bld 3)^{\sqcup
      i}}(*)]$ represents the suspension $\Sigma^{2i}\bS$ as an object
  in the stable homotopy category and the map represents the $i$-fold
  product $u^i$ in $\pi_{2i}(ku_{(p)})$. Hence it follows that the
  composite map
\begin{equation}\label{eq:p-local-ku-splitting}
\bigvee_{i=0}^{p-2}\bS^{\cJ}[F^{\cJ}_{(\bld 1,\bld 3)^{\sqcup i}}(*)]\wedge \ell' \to ku_{(p)}\wedge ku_{(p)}\to ku_{(p)}
\end{equation}
is a stable equivalence of $\ell'$-module spectra.

Using the $\cJ$-equivalences $D(v) \to \C{v}^{\gp}_{\geq 0}$ and $D(u)
\to \C{u}^{\gp}_{\geq 0}$ from Lemma~\ref{lem:hty-type-of-Dx}, and
Lemma~\ref{lem:free-p-local-homology}, it follows that
$H_*(\bS^{\cJ}[D(u)];\bZ_{(p)})$ is a free
$H_*(\bS^{\cJ}[D(v)];\bZ_{(p)})$-module.  Inspecting
Lemma~\ref{lem:free-p-local-homology}, we see that it is generated by
the images of the canonical generators of $\Sigma^{2i}\bZ_{(p)}$ under
the homomorphisms
\[
\Sigma^{2i}\bZ_{(p)}\cong H_*(\bS^{\cJ}[F^{\cJ}_{(\bld 1,\bld 3)^{\sqcup i}}(*)];\bZ_{(p)})
\to H_*(\bS^{\cJ}[\C{u}];\bZ_{(p)}) \to H_*(\bS^{\cJ}[D(u)];\bZ_{(p)})
\]
for $i=0,\dots,p-2$. Hence the
$\Tor$ spectral sequence \cite{EKMM}*{Theorem~IV.4.1}
\[
\mathrm{Tor}_{s,t}^{E_*(\bS^{\cJ}[D(v)])}(E_*(\bS^{\cJ}[D(u)]), E_*(\ell'))
\Longrightarrow E_{s+t}(\bS^{\cJ}[D(u)]\wedge_{\bS^{\cJ}[D(v)]}\ell')
\]
with $E=H\bZ_{(p)}$ collapses. The result follows from this since the
specified generators are compatible with the stable equivalence in
\eqref{eq:p-local-ku-splitting}.
\end{proof}

\section{Logarithmic \texorpdfstring{$\THH$}{THH} of the Adams summand}\label{sec:log-THH-ell}
We will now demonstrate that the current definition of logarithmic
topological Hochschild homology, $\THH(A, M)$, in terms of $\THH(A)$
and the cyclic and replete bar constructions, lends itself to nontrivial
explicit computations, going beyond the cases of discrete rings previously
studied by Hesselholt--Madsen \cite{Hesselholt-M_local_fields}*{\S2}.  In particular, we will
realize the program to compute $V(1)_* \THH(ku) \cong V(1)_* \THH(\kup)$
outlined by Ausoni in \cite{Ausoni_THH-ku}*{\S10}, using $\THH(\ell, D(v))$
and $\THH(\kup, D(u))$ in place of the then-hypothetical constructions
$\THH(\ell|L)$ and $\THH(\kup|\KUp)$.  In Theorem~\ref{thm:ausoni} we
use this to recover the full algebra structure on $V(1)_* \THH(ku)$,
for $p\ge5$.  The subsequent construction of $\THH(\ell|L)$ and
$\THH(\kup|\KUp)$ by Blumberg--Mandell~\cite{Blumberg-M_loc-sequenceTHH}, using simplicially
enriched Waldhausen categories, is not known to lend itself to such
explicit calculations.  On the other hand, their models are known to admit
good trace maps from algebraic $K$-theory, so one may hope to prove that
the two constructions are equivalent, in the cases where both are defined.

\begin{notation}\label{nota:Smith-Toda-complex}
For any prime~$p$, let $H = H\bF_p$ be the mod~$p$ Eilenberg--Mac\,Lane
spectrum and write $H_* X = \pi_*(H \wedge X)$ for the mod~$p$ homology
groups of~$X$.  For $p\ge3$ let
\[
V(1) = \cone(v_1 \: \Sigma^{2p-2} S/p \to S/p)
	\simeq S \cup_p e^1 \cup_{\alpha_1} e^{2p-1} \cup_p e^{2p}
\]
be the Smith--Toda complex of type~$2$.  When $p\ge5$ it admits the
structure of a homotopy commutative ring
spectrum~\cite{Oka_few-cells}, and $V(1) \to H$ is a map of homotopy
commutative ring spectra.  Write $V(1)_* X = \pi_*(V(1) \wedge X)$ for
the $V(1)$-homotopy groups of~$X$.

When $X$ is a ring spectrum and $p\ge5$, $V(1)_* X$ becomes a graded
$\bF_p$-algebra, and $V(1)_*X \to H_*X$ is an algebra homomorphism.
When $X$ has a circle action $S^1_+ \wedge X \to X$, $V(1)_*X$
inherits a suspension operator \[\sigma \: V(1)_* X \to
V(1)_{*+1}(S^1_+ \wedge X) \to V(1)_{*+1} X,\] compatible with the
operator $\sigma \: H_*X \to H_{*+1}(S^1_+ \wedge X) \to H_{*+1} X$.
When the adjoint circle action $X \to F(S^1_+, X)$ is a ring spectrum
map, both suspension operators are
derivations~\cite[Proposition~5.10]{Angeltveit-R_Hopf-algebra}.

There is an equivalence $V(1) \wedge \ell \simeq H$ of homotopy
commutative $\ell$-algebras.  When $X$ is an $\ell$-module or
commutative $\ell$-algebra there is an equivalence $V(1) \wedge X
\simeq H \wedge_\ell X$ of spectra or homotopy commutative ring
spectra, respectively, and a corresponding isomorphism $V(1)_* X \cong
\pi_*(H \wedge_\ell X)$. (Smash products are understood as left derived
smash products here.) In particular, the natural map $V(1)_* X \to
H_* X$ is split injective. When $p=3$ and $X$ is a commutative
$\ell$-algebra we give $V(1)_* X$ the algebra structure from $\pi_*(H
\wedge_\ell X)$.
\end{notation}

Let $(A, M)$ be a pre-log ring spectrum with $M = M_{\ge0}$ concentrated
in non-negative degrees.  Suppose, without loss of generality, that
$(A, M)$ is cofibrant, so that the maps $\bS^{\cJ}[M] \to A$ and
$\bS^{\cJ}[B^{\cy}(M)] \to \THH(A)$ are cofibrations of commutative
symmetric ring spectra.  This ensures that the smash
product
\[
\THH(A, M) = \THH(A) \wedge_{\bS^{\cJ}[B^{\cy}(M)]} \bS^{\cJ}[B^{\rep}(M)]
\]
captures the correct homotopy type.

In order to determine the structure of a K{\"u}nneth
spectral sequence associated to this smash product, we shall
use a natural chain of $\bS^\cJ[B^{\cy}(M)]$-module maps
\[
\bS^\cJ[B^{\rep}(M)] \longto
\bS^\cJ[B^{\cy}_{\{0\}}(M^{\gp})] \longleftarrow
\bS^\cJ[B^{\cy}(M_{\{0\}})] \,.
\]
The left hand map is defined by first projecting onto $\bS^\cJ[B^{\rep}_{\{0\}}(M)]$ as in \cite[Definition~6.9]{RSS_LogTHH-I} and then composing with the canonical map to $\bS^\cJ[B^{\cy}_{\{0\}}(M^{\gp})]$, while the right hand map uses 
the identification $B^{\cy}(M_{\{0\}}) = B^{\cy}_{\{0\}}(M)$ and the group completion
map $M \to M^{\gp}$.
Hence there is a chain of $\THH(A)$-module maps
\begin{equation}\label{eq:THHAM-THHAmommodM}
\THH(A, M) \longto \THH(A) \wedge_{\bS^\cJ[B^{\cy}(M)]}
	\bS^\cJ[B^{\cy}_{\{0\}}(M^{\gp})]
\longleftarrow \THH(\AmodmodM) \,,
\end{equation}
where $\AmodmodM = A \wedge_{\bS^\cJ[M]} \bS^\cJ[M_{\{0\}}]$.

We shall use a corresponding chain of K{\"u}nneth spectral sequences to
transport information about the spectral sequence for $\THH(\AmodmodM)$
to the spectral sequence for $\THH(A, M)$.  This chain has the
following $E^2$-terms:
\begin{multline*}
\Tor^{H_* \bS^\cJ[B^{\cy}(M)]}_{**}
	(V(1)_* \THH(A), H_* \bS^{\cJ}[B^{\rep}(M)]) \\
\longto
\Tor^{H_* \bS^\cJ[B^{\cy}(M)]}_{**}
	(V(1)_* \THH(A), H_* \bS^{\cJ}[B^{\cy}_{\{0\}}(M^{\gp})]) \\
\longleftarrow
\Tor^{H_* \bS^\cJ[B^{\cy}(M)]}_{**}
	(V(1)_* \THH(A), H_* \bS^{\cJ}[B^{\cy}(M_{\{0\}})])
\end{multline*}
and converges to the $V(1)$-homotopy of the chain of $\THH(A)$-modules
displayed above. It will be constructed in the proof of Theorem~\ref{thm:thhelldv}.
The reader may want to compare the following calculations
with those in~\cite[Section~5]{RSS_LogTHH-I}, where we handled the case of a discrete
pre-log structure.

Our assumption that $(A,M)$ is cofibrant implies that $A$ is a cofibrant commutative symmetric ring spectrum. Now suppose in addition that $A$ is augmented over $H \bF_p$, and let $H \to H\bF_p$ be a cofibrant replacement in commutative $A$-algebras. If~$X$ is an $A$-module, we write $H^A_*(X) = \pi_*(H\sm_A X)$.  Using the isomorphism
\[
\THH(A) \wedge_{\bS^\cJ[B^{\cy}(M)]} \bS^\cJ[B^{\cy}(M_{\{0\}})]
	\iso \THH(A/(M_{>0}))
\]
and the maps $\bS \to A\to H$ we get a pushout square
of commutative $H$-algebras
\[
\xymatrix@-1pc{
H \wedge \bS^\cJ[B^{\cy}(M)] \ar[rr] \ar[d]
	&& H \wedge \bS^\cJ[B^{\cy}(M_{\{0\}})] \ar[d] \\
H \wedge_{A} \THH(A) \ar[rr]
	&& H \wedge_{A} \THH(A/(M_{>0}))\ .
}
\]
which is homotopy cocartesian by our cofibrancy assumptions.  We first study the associated
$\Tor$ spectral sequence 
\begin{equation} \label{eq:thhamgt0ss}
\begin{aligned}
E^2_{**} &= \Tor^{H_* \bS^\cJ[B^{\cy}(M)]}_{**}
	(H^A_* \THH(A), H_* \bS^\cJ[B^{\cy}(M_{\{0\}})]) \\
&\Longrightarrow H^A_* \THH(A/(M_{>0})) \,.
\end{aligned}
\end{equation}
We shall use the notation $d^r(x) \doteq y$ to
indicate that $d^r(x)$ equals a unit in $\bF_p$ times~$y$.
\begin{proposition} \label{prop:thhz}
Consider the case $A = \ell$, $M = D(v)$ and $\AmodmodM \simeq H\bZ_{(p)}$
of the spectral sequence~\eqref{eq:thhamgt0ss} above. It is an algebra spectral sequence
\begin{align*}
E^2_{**} &= \Tor^{H_\circledast (B^{\cy} (D(v))_{h\cJ})}_{**}
	(V(1)_* \THH(\ell), H_\circledast (B^{\cy}(D(v))_{h\cJ_{0}})) \\
	&\Longrightarrow V(1)_* \THH(\bZ_{(p)}) \,. 
\end{align*}
With $C_* = H_*(B^{\cy}(D(v))_{h\cJ_{0}}) $ as in~\eqref{eq:def-C-HBcyDx0} we have 
\begin{align*}
H_\circledast (B^{\cy}(D(v))_{h\cJ}) &\cong P(v) \otimes E(dv) \otimes C_* \\
V(1)_* \THH(\ell) &\cong E(\lambda_1, \lambda_2) \otimes P(\mu_2) \\
V(1)_* \THH(\bZ_{(p)}) &\cong E(\epsilon_1, \lambda_1) \otimes P(\mu_1) \,,
\end{align*}
with $|v| = 2p-2$, $|dv| = 2p-1$, $|\lambda_1| = 2p-1$, $|\lambda_2|
= 2p^2-1$, $|\mu_2| = 2p^2$, $|\epsilon_1| = 2p-1$ and $|\mu_1| = 2p$.
Here
\[
E^2_{**} \cong E(\lambda_1, \lambda_2) \otimes P(\mu_2)
	\otimes E([v]) \otimes \Gamma([dv])
\]
where $[v]$ has bidegree~$(1,2p-2)$ and $[dv]$ has bidegree~$(1,2p-1)$.
There are nontrivial $d^p$-differentials
\[
d^p(\gamma_k[dv]) \doteq \lambda_2 \cdot \gamma_{k-p}[dv]
\]
for all $k\ge p$, leaving
\[
E^\infty_{**} \cong E(\lambda_1) \otimes P(\mu_2) \otimes E([v])
	\otimes P_p([dv]) \,.
\]
Hence $[v]$ represents $\epsilon_1$ (modulo~$\lambda_1$), $[dv]$
represents $\mu_1$ and $\mu_2$ represents $\mu_1^p$ (up to units in
$\bF_p$) in the abutment, and there is a multiplicative extension $[dv]^p
\doteq \mu_2$.
\end{proposition}

\begin{proof}
  Building on the graded Thom isomorphism, Proposition~\ref{prop:homology-B^cyDx} provides isomorphisms 
$H_* \bS^\cJ[B^{\cy}_{\{0\}}(D(v))]
  \cong H_* B^{\cy}(D(v)_{\{0\}})_{h\cJ} = C_*$ and $H_* \bS^\cJ[B^{\cy}(D(v))] \cong H_\circledast (B^{\cy} (D(v))_{h\cJ}) \cong P(v) \otimes E(dv)
  \otimes C_*$.

   B{\"o}kstedt computed $\pi_*(S/p \wedge\THH(\bZ))
   \cong \pi_*(S/p \wedge\THH(\bZ_{(p)})) = E(\lambda_1) \otimes
   P(\mu_1)$, so we have $V(1)_* \THH(\bZ_{(p)}) = E(\epsilon_1,
   \lambda_1) \otimes P(\mu_1)$, where $\epsilon_1$ is a mod~$v_1$
   Bockstein element in degree~$2p-1$.  McClure and Staffeldt
   \cite{McClure-S_thh-bu} computed that $V(1)_* \THH(\ell) =
   E(\lambda_1, \lambda_2) \otimes P(\mu_2)$.  See
   \cite{Ausoni-R_K-Morava}*{\S3, \S4} for further details.

This leads to the $E^2$-term
\begin{align*}
E^2_{**} &= \Tor^{P(v) \otimes E(dv) \otimes C_*}_{**}
	(E(\lambda_1, \lambda_2) \otimes P(\mu_2), C_*) \\
&\cong \Tor^{P(v) \otimes E(dv)}_{**}
	(E(\lambda_1, \lambda_2) \otimes P(\mu_2), \bF_p) \\
&\cong E(\lambda_1, \lambda_2) \otimes P(\mu_2)
	\otimes E([v]) \otimes \Gamma([dv]) \,,
\end{align*}
where we have used change-of-rings and the fact that $P(v) \otimes E(dv)$
acts trivially on $E(\lambda_1, \lambda_2) \otimes P(\mu_2)$.
To verify the last assertion, we use the factorization \[H_* \bS^\cJ[B^{\cy}(D(v))]
\to H_* \THH(\ell) \to H^\ell_*(\THH(\ell))  \cong V(1)_* \THH(\ell)\,.\]
The first homomorphism extends $H_* \bS^\cJ[D(v)] \to H_* \ell$, hence
takes $v$ to $0$ since~$v\in \pi_{2p-2}(\ell)$ has Adams filtration~$1$, and takes $dv$
to $0$ by compatibility with the suspension operator coming from the
circle action.

The algebra generators $\lambda_1$, $\lambda_2$, $\mu_2$, $[v]$ and
$[dv]$ must be infinite cycles for filtration reasons.  To determine
the differentials on the remaining algebra generators, namely the
divided powers $\gamma_{p^i}[dv]$ for $i\ge1$, we note that the abutment
$E(\epsilon_1, \lambda_1) \otimes P(\mu_1)$ has at most two generators
in each degree.  In total degree~$2p^2-1$ the $E^2$-term is generated by
the three classes $\lambda_2$, $\lambda_1 \cdot \gamma_{p-1}[dv]$ and
$[v] \cdot \gamma_{p-1}[dv]$.  Hence one of these must be a boundary,
and for filtration reasons the only possibility is $d^p(\gamma_p[dv])
\doteq \lambda_2$.

This implies that $\lambda_1$, $[v]$, $[dv]$ and $\mu_2$ survive to the
$E^\infty$-term, where they must represent $\lambda_1$, $\epsilon_1$
(modulo $\lambda_1$), $\mu_1$ and $\mu_1^p$, respectively (up to units
in $\bF_p$).  It follows that each generating monomial in $E(\lambda_1)
\otimes P(\mu_2) \otimes E([v]) \otimes P_p([dv])$ is an infinite cycle
that represents a nonzero product in the abutment, so these classes
cannot be boundaries.

Now consider total degree~$2p^3-1$.  After the differential on
$\gamma_p[dv]$, only the three generators
\[
\lambda_1 \cdot \mu_2^{p-1} \cdot \gamma_{p-1}[dv]
\quad\text{,}\quad
\mu_2^{p-1} \cdot [v] \cdot \gamma_{p-1}[dv]
\quad\text{and}\quad
\lambda_2 \cdot \gamma_{p^2-p}[dv]
\]
remain.  As we have just noticed, the first two monomials cannot be hit
by differentials.  Since only two generators can survive in this
degree, and the only possible source (or target) of a differential is $\gamma_{p^2}[dv]$,
we must have $d^p(\gamma_{p^2}[dv]) \doteq \lambda_2 \cdot \gamma_{p^2-p}[dv]$.
By induction, the corresponding argument in degree~$2p^{i+1}-1$
establishes the nontrivial differential on $\gamma_{p^i}[dv]$, for each
$i\ge2$.
\end{proof}
Analogously to the homotopy cocartesian square leading
to~\eqref{eq:thhamgt0ss}, the smash product defining $\THH(A, M)$ gives
rise to a homotopy cocartesian square
\[\xymatrix@-1pc{
H \wedge \bS^\cJ[B^{\cy}(M)] \ar[rr] \ar[d]
	&& H \wedge \bS^\cJ[B^{\rep}(M)] \ar[d] \\
H \wedge_A \THH(A) \ar[rr]
	&& H \wedge_A \THH(A, M)
}
\]
of commutative $H$-algebras, and an associated $\Tor$ spectral sequence
\begin{equation} \label{eq:logthhamss}
\begin{aligned}
E^2_{**} &= \Tor^{H_* \bS^\cJ[B^{\cy}(M)]}_{**}
        (H^A_* \THH(A), H_* \bS^\cJ[B^{\rep}(M)]) \\
&\Longrightarrow H^A_* \THH(A, M) \,.
\end{aligned}
\end{equation}

\begin{theorem} \label{thm:thhelldv}
Consider the case $A = \ell$ and $M = D(v)$ of the spectral
sequence~\eqref{eq:logthhamss} above.  It is an algebra spectral
sequence
\begin{align*}
E^2_{**} &= \Tor^{H_\circledast (B^{\cy} (D(v))_{h\cJ})}_{**}
	(V(1)_* \THH(\ell), H_\circledast (B^{\rep}(D(v))_{h\cJ})) \\
	&\Longrightarrow V(1)_* \THH(\ell, D(v)) \,,
\end{align*}
where
\begin{align*}
H_\circledast (B^{\cy}(D(v))_{h\cJ}) &= P(v) \otimes E(dv) \otimes C_* \\
H_\circledast (B^{\rep}(D(v))_{h\cJ}) &= P(v) \otimes E(d\log v) \otimes C_* \\
V(1)_* \THH(\ell) &= E(\lambda_1, \lambda_2) \otimes P(\mu_2) \,,
\end{align*}
with $|d\log v| = 1$ and the remaining degrees as above.  Here
\[
E^2_{**} = E(\lambda_1, \lambda_2) \otimes P(\mu_2)
	\otimes E(d\log v) \otimes \Gamma([dv])
\]
where $[dv]$ has bidegree~$(1,2p-1)$.  There are nontrivial differentials
\[
d^p(\gamma_k[dv]) \doteq \lambda_2 \cdot \gamma_{k-p}[dv]
\]
for all $k\ge p$, leaving
\[
E^\infty_{**} = E(\lambda_1) \otimes P(\mu_2) \otimes E(d\log v)
	\otimes P_p([dv]) \,.
\]
There is a multiplicative extension $[dv]^p \doteq \mu_2$, so the abutment
is
\[
V(1)_* \THH(\ell, D(v)) = E(\lambda_1, d\log v) \otimes P(\kappa_1) \,,
\]
where $\kappa_1$ is represented by $[dv]$ in degree~$2p$.
\end{theorem}
Together with the stable equivalence $\THH(\ell, D(v)) \to \THH(\ell,
j_*\!\GLoneJof(L))$ from Proposition~\ref{prop:logification-of-D-of-x},
and Proposition~\ref{prop:circle-operator-values} below, the previous
theorem implies Theorem~\ref{thm:V(1)THHellGLoneJofL-intro} from the
introduction. Note that $\kappa_1 \in V(1)_{2p} \THH(\ell, D(v))$ is
only defined modulo~$\lambda_1 \cdot d\log v$.
\begin{proof}[Proof of Theorem~\ref{thm:thhelldv}]
  Recall the chain~\eqref{eq:THHAM-THHAmommodM}. We apply the same
  cobase changes as before, and get a chain of three $\Tor$ spectral
  sequences with $E^2$-terms
\begin{multline*}
\Tor^{P(v) \otimes E(dv) \otimes C_*}_{**}
	(E(\lambda_1, \lambda_2) \otimes P(\mu_2),
	P(v) \otimes E(d\log v) \otimes C_*) \\
\longto
\Tor^{P(v) \otimes E(dv) \otimes C_*}_{**}
	(E(\lambda_1, \lambda_2) \otimes P(\mu_2),
	E(d\log v) \otimes C_*) \\
\longleftarrow
\Tor^{P(v) \otimes E(dv) \otimes C_*}_{**}
	(E(\lambda_1, \lambda_2) \otimes P(\mu_2),
	C_*) \,,
\end{multline*}
which by change-of-rings is isomorphic to the chain
\begin{multline*}
\Tor^{E(dv)}_{**}
	(E(\lambda_1, \lambda_2) \otimes P(\mu_2), E(d\log v)) \\
\longto
\Tor^{P(v) \otimes E(dv)}_{**}
	(E(\lambda_1, \lambda_2) \otimes P(\mu_2), E(d\log v)) \\
\longleftarrow
\Tor^{P(v) \otimes E(dv)}_{**}
	(E(\lambda_1, \lambda_2) \otimes P(\mu_2), \bF_p) \,,
\end{multline*}
hence takes the form
\begin{multline*}
E(\lambda_1, \lambda_2) \otimes P(\mu_2) \otimes E(d\log v)
	\otimes \Gamma([dv]) \\
\longto
E(\lambda_1, \lambda_2) \otimes P(\mu_2) \otimes E(d\log v)
	\otimes E([v]) \otimes \Gamma([dv]) \\
\longleftarrow
E(\lambda_1, \lambda_2) \otimes P(\mu_2)
	\otimes E([v]) \otimes \Gamma([dv]) \,,
\end{multline*}
and converges to the chain
\begin{multline*}
V(1)_* \THH(\ell, D(v)) \\
	\longto V(1)_*(\THH(\ell)
	\wedge_{\bS^\cJ[B^{\cy}(D(v))]} \bS^\cJ[B^{\cy}_{\{0\}}(D(v)^{\gp})]) \\
	\longleftarrow V(1)_* \THH(\bZ_{(p)}) \,.
\end{multline*}

The known differentials $d^p(\gamma_k[dv]) \doteq \lambda_2 \cdot
\gamma_{k-p}[dv]$ in the right hand spectral sequence (which is that
of Proposition~\ref{prop:thhz}) remain nonzero in the middle spectral
sequence, since the right hand arrow of $E^2$-terms is injective and
$d\log v$ is an infinite cycle for filtration reasons.  Likewise the
multiplicative extension $[dv]^p \doteq \mu_2$ carries over to the
middle.  Furthermore, since the left hand arrow of $E^2$-terms is also
injective, it follows that the differentials and multiplicative
extensions lift to the left hand spectral sequence.

This implies that the asserted $d^p$-differentials are the first nonzero
differentials in the left hand spectral sequence, which is the spectral
sequence in the statement of the theorem.  This leaves the
$E^{p+1}$-term
\[
E^{p+1}_{**} = E(\lambda_1) \otimes P(\mu_2) \otimes E(d\log v)
	\otimes P_p([dv]) \,,
\]
which must be equal to the $E^\infty$-term for filtration reasons.
Letting $\kappa_1$ in degree~$2p$ be a class in the abutment represented
by $[dv]$ in bidegree~$(1,2p-1)$, we find that $\kappa_1^p \doteq \mu_2$,
leading to the asserted algebraic structure of the abutment.
\end{proof}

By Theorem~\ref{thm:localization-seq-for-Dx} we have a homotopy
cofiber sequence
\[
\THH(\ell) \overset{\rho}\longto
\THH(\ell, D(v)) \overset{\partial}\longto
\Sigma \THH(\bZ_{(p)})
\]
of $\THH(\ell)$-modules, where $\rho$ is a map of commutative
symmetric ring spectra.  Let $\tau \: \THH(\bZ_{(p)}) \to \THH(\ell)$
denote the homotopy fiber map of $\rho$, which we like to think of as a kind of
transfer map.

\begin{lemma} \label{lem:v1exseqelldv}
There is a long exact sequence
\begin{multline*}
\dots \longto V(1)_* \THH(\bZ_{(p)}) \overset{\tau_*}\longto
	V(1)_* \THH(\ell) \\
\overset{\rho_*}\longto V(1)_* \THH(\ell, D(v))
	\overset{\partial_*}\longto V(1)_{*-1} \THH(\bZ_{(p)})
	\longto \dots
\end{multline*}
of $V(1)_* \THH(\ell)$-modules, where $\rho_*$ is an algebra homomorphism,
and
\begin{enumerate}
\item
$\tau_*(\epsilon_1 \mu_1^{p-1}) \doteq \lambda_2$,
\item
$\rho_*(\lambda_1) = \lambda_1$,
\item
$\rho_*(\mu_2) \doteq \kappa_1^p$,
\item
$\partial_*(d\log v \cdot \kappa_1^k) \doteq \mu_1^k$ for $k\ge0$ and
\item
$\partial_*(\kappa_1^k) \doteq \epsilon_1 \mu_1^{k-1}
\pmod{\lambda_1 \mu_1^{k-1}}$ for $p\nmid k\ge1$.
\end{enumerate}
\end{lemma}

\begin{proof}
The repletion homomorphism
\[
\rho_* \: E(\lambda_1, \lambda_2) \otimes P(\mu_2)
\longto E(\lambda_1, d\log v) \otimes P(\kappa_1)
\]
is induced by the canonical $H$-algebra map
\[
H \wedge_\ell \THH(\ell) \longto H \wedge_\ell \THH(\ell, D(v)) \,,
\]
and therefore factors through the edge homomorphism of the
spectral sequence in Theorem~\ref{thm:thhelldv}. Hence it is an
algebra homomorphism satisfying $\rho_*(\lambda_1) = \lambda_1$,
$\rho_*(\lambda_2) = 0$ and $\rho_*(\mu_2) \doteq \kappa_1^p$.

It follows that
\[
\ker(\rho_*) = \im(\tau_*) = E(\lambda_1) \otimes P(\mu_2) \{\lambda_2\}
\]
is the free $E(\lambda_1) \otimes P(\mu_2)$-submodule of $V(1)_*
\THH(\ell)$ generated by $\lambda_2$.  Only (a unit in $\bF_p$ times)
$\epsilon_1 \lambda_1 \mu_1^{p-1}$ can map under $\tau_*$ to $\lambda_1
\lambda_2$, so $\tau_*(\epsilon_1 \mu_1^{p-1}) \doteq \lambda_2$.
Likewise
\[
\im(\rho_*) = \ker(\partial_*) = E(\lambda_1) \otimes P(\kappa_1^p)
\]
is the free $E(\lambda_1) \otimes P(\mu_2)$-submodule of $V(1)_*
\THH(\ell, D(v))$ generated by $1$.
The classes $d\log v \cdot \kappa_1^k$ can only map non-trivially under
$\partial_*$ to (units in $\bF_p$ times) $\mu_1^k$, so $\partial_*(d\log v
\cdot \kappa_1^k) \doteq \mu_1^k$ for $k\ge0$.  The classes $\kappa_1^k$
with $p\nmid k\ge1$ must map to classes in the span of $\epsilon_1
\mu_1^{k-1}$ and $\lambda_1 \mu_1^{k-1}$ that are linearly independent
of \[\partial_*(\lambda_1 \cdot d\log v \cdot \kappa_1^{k-1}) \doteq
\lambda_1 \mu_1^{k-1},\] hence must map to (a unit in $\bF_p$ times)
$\epsilon_1 \mu_1^{k-1}$ modulo a multiple of $\lambda_1 \mu_1^{k-1}$.
\end{proof}

\begin{lemma} \label{lem:prelogalpha}
The adjoint pre-log structure map
\[
\bar\alpha \: \bS^\cJ[B^{\rep} D(v)] \longto \THH(\ell, D(v))
\]
induces an algebra homomorphism
\[
\bar\alpha_* \: H_\circledast(B^{\rep}(D(v))_{h\cJ})
	\longto V(1)_* \THH(\ell, D(v))
\]
satisfying $\bar\alpha_*(v) = 0$ and $\bar\alpha_*(d\log v) = d\log v$.
\end{lemma}

\begin{proof}
The adjoint pre-log structure homomorphism
\[
\bar\alpha_* \: P(v) \otimes E(d\log v) \otimes C_*
\longto E(\lambda_1, d\log v) \otimes P(\kappa_1)
\]
is induced by the canonical $H$-algebra map
\[
H \wedge \bS^\cJ[B^{\rep} D(v)] \longto H \wedge_\ell \THH(\ell, D(v)) \,,
\]
and therefore
factors through the edge homomorphism of the spectral sequence in
Theorem~\ref{thm:thhelldv}. Hence it is an algebra homomorphism
satisfying $\bar\alpha_*(v) = 0$ and $\bar\alpha_*(d\log v) = d\log
v$.
\end{proof}

\section{Logarithmic \texorpdfstring{$\THH$}{THH} of the connective complex \texorpdfstring{$K$}{K}-theory spectrum}\label{sec:log-THH-ku}
We now return to the setup of Section~\ref{sec:elog-etaleness} and
consider the tamely ramified map $f \: \ell \to \kup$, inducing $f_*
\: \bZ_{(p)}[v] \to \bZ_{(p)}[u]$ with $f_*(v) = u^{p-1}$ in homotopy,
and its log $\THH$-{\'e}tale extension
\[
(f, f^\flat) \: (\ell, D(v), \alpha) \to (\kup, D(u), \beta)
\]
where $\alpha \: D(v) \to \Omega^\cJ (\ell)$ and $\beta \: D(u) \to
\Omega^\cJ (\kup)$ denote the Adams and Bott pre-log structures on
$\ell$ and $\kup$, respectively, and $f^\flat(v) = u^{p-1}$. We note
that the induced map of residue ring spectra
\[
f/(f^\flat_{>0}) \:
	\ell/(D(v)_{>0}) \overset{\sim}\longto \kup/(D(u)_{>0})
\]
is a stable equivalence, with both sides equivalent to $H\bZ_{(p)}$.
Therefore Theorem~\ref{thm:localization-seq-for-Dx} provides a diagram
of homotopy cofiber sequences
\begin{equation} \label{eq:cofseq-ludv-kudu}
\xymatrix@-1pc{
\THH(\ell) \ar[r]^-{\rho} \ar[d]_f&
\THH(\ell, D(v)) \ar[r]^-{\partial} \ar[d]_{(f,f^{\flat})}&
\Sigma \THH(\bZ_{(p)}) \ar@{=}[d]\\
\THH(\kup) \ar[r]^-{\rho'} & \THH(\kup, D(u)) \ar[r]^-{\partial'} & \Sigma \THH(\bZ_{(p)})
}
\end{equation}
where the left hand square is strictly commutative and the right hand square is homotopy commutative.  We have proved in
Theorem~\ref{thm:ell-ku-Du-Dv-logthh-etale} that the induced map
\begin{equation} \label{eq:lkuthhetale}
\kup \wedge_\ell \THH(\ell, D(v))
	\overset{\sim}\longto \THH(\kup, D(u))
\end{equation}
is a stable equivalence, where here the smash product over $\ell$ should be understood as a left derived smash product.

The stable equivalence $\THH(\kup, D(u)) \to \THH(\kup,
j_*\!\GLoneJof(KU_{(p)}))$ from
Proposition~\ref{prop:logification-of-D-of-x},
Proposition~\ref{prop:circle-operator-values} below, and the following
theorem imply Theorem~\ref{thm:V(1)THHkuGLoneJofKU-intro} from the
introduction.

\begin{theorem} \label{thm:thhkudu}
There is an algebra isomorphism
\[
V(1)_* \THH(\kup, D(u)) \cong P_{p-1}(u) \otimes E(\lambda_1, d\log u)
        \otimes P(\kappa_1)
\]
with $|u| = 2$, $|\lambda_1| = 2p-1$, $|d\log u| = 1$ and $|\kappa_1|
= 2p$. The cobase change
equivalence~\eqref{eq:lkuthhetale} induces the isomorphism
\[
P_{p-1}(u) \otimes E(\lambda_1, d\log v) \otimes P(\kappa_1)
\overset{\cong}\longto
P_{p-1}(u) \otimes E(\lambda_1, d\log u) \otimes P(\kappa_1)
\]
that maps $d\log v$ to $-d\log u$ and preserves the other terms.
The adjoint pre-log structure map
\[
\bar\beta \: \bS^\cJ[B^{\rep}(D(u))] \longto \THH(\kup, D(u))
\]
induces an algebra homomorphism
\[
\bar\beta_* \: H_\circledast(B^{\rep}(D(u))_{h\cJ})
	\longto V(1)_* \THH(\kup, D(u))
\]
satisfying $\bar\beta_*(u) = u$ and $\bar\beta_*(d\log u) = d\log
u$. The suspension operator satisfies $\sigma(u) = u \cdot d\log u$ and $\sigma(d\log u) = 0$.
\end{theorem}

\begin{proof}
The $\Tor$ spectral sequence
\begin{align*}
E^2_{**} &= \Tor^{\pi_* \ell}_{**}
	(\pi_* \kup, V(1)_* \THH(\ell, D(v))) \\
	&\Longrightarrow V(1)_* \THH(\kup, D(u))
\end{align*}
collapses at the $E^2$-term in filtration~$0$, since $\pi_* \kup$
is a free $\pi_* \ell$-module and $V(1)_* \THH(\ell, D(v))$ is a trivial
$\pi_* \ell$-module.  The term $P_{p-1}(u)$ arises as \[\pi_* \kup
\otimes_{\pi_* \ell} \bF_p \cong P(u) \otimes_{P(v)} \bF_p\,.\]
Chasing the class $u$ in $\cJ$-degree~$2$ around the commutative square
\[
\xymatrix@-1pc{
H_\circledast(D(u)_{h\cJ}) \ar[rr] \ar[d]
	&& V(1)_* \kup \ar[d] \\
H_\circledast(B^{\rep}(D(u)_{h\cJ})) \ar[rr]^-{\bar\beta_*}
	&& V(1)_* \THH(\kup, D(u))
}
\]
we see that $\bar\beta_*(u) = u$.

Chasing the class $d\log v$ in $\cJ$-degree~$0$ around the commutative
square
\[
\xymatrix@-1pc{
H_\circledast (B^{\rep}(D(v))_{h\cJ}) \ar[rr]^-{\bar\alpha_*} \ar[d]_-{f^\flat_\circledast}
	&& V(1)_* \THH(\ell, D(v)) \ar[d]^-{f_*} \\
H_\circledast (B^{\rep}(D(u))_{h\cJ}) \ar[rr]^-{\bar\beta_*}
	&& V(1)_* \THH(\kup, D(u))
}
\]
we find that $f^\flat_\circledast(d\log v) = (p-1) d\log u = -d\log u$
maps under $\bar\beta_*$ to the image of $\bar\alpha_*(d\log v)
= d\log v$ under $f_*$.  Hence we can trade $d\log u$ for
$d\log v$ as a generator in $V(1)_* \THH(\kup, D(u))$,
giving the asserted formulas. 

The $(2p-3)$-connected map $V(1) \to H$ induces a commutative
diagram
\[
\xymatrix@-1pc{
H_\circledast (B^{\cy}(D(u))_{h\cJ}) \ar[d]_{\rho_\circledast}
& V(1)_* \bS^{\cJ}[B^{\cy}(D(u))] \ar[d]_{\rho_*}
	\ar[l] \ar[r]^-{\bar\beta_*}
& V(1)_* \THH(ku_{(p)}) \ar[d]^{\rho'_*} \\
H_\circledast (B^{\rep}(D(u))_{h\cJ})
& V(1)_* \bS^{\cJ}[B^{\rep}(D(u))] \ar[l] \ar[r]^-{\bar\beta_*}
& V(1)_* \THH(ku_{(p)}, D(u)) \,,
}
\]
where the left hand horizontal arrows are $(2p-3)$-connected.  By
Proposition~\ref{prop:homology-B^cyDx} we have $\rho_\circledast(du) =
u \cdot d\log u = \sigma(u)$ at the left hand side.  By the
connectivity estimate we have $\rho_*(du) = u \cdot d\log u =
\sigma(u)$ in the middle (modulo $\alpha_1 \in \pi_3 V(1)$ for $p=3$),
which implies that $\rho'_*(du) = u \cdot d\log u = \sigma(u)$ at the
right hand side. By Proposition~\ref{prop:homology-B^cyDx} we also
have $\sigma(d\log u) = 0$ at the left hand side, so by the same
connectivity estimate we have $\sigma(d\log u) = 0$ in the middle and
at the right hand side.
\end{proof}

In the next lemma we consider the square obtained by applying the left derived cobase change along $\ell\to ku_{(p)}$ to the right hand square in \eqref{eq:cofseq-ludv-kudu}. We write 
\[
\chi\colon \kup \wedge_\ell \Sigma\THH(\bZ_{(p)}) \to \Sigma\THH(\bZ_{(p)})
\]
 for the induced map of $\kup$-modules.

\begin{lemma} \label{lem:chipartial}
The homotopy commutative square
\[
\xymatrix@-1pc{
\kup \wedge_\ell \THH(\ell, D(v))
	\ar[rr]^-{1\wedge\partial} \ar[d]_-{\sim}
	&& \kup \wedge_\ell \Sigma\THH(\bZ_{(p)}) \ar[d]^{\chi} \\
\THH(\kup, D(u)) \ar[rr]^-{\partial'} && \Sigma\THH(\bZ_{(p)})
}
\]
induces a commutative square
\[
\xymatrix@-1pc{
P_{p-1}(u) \otimes V(1)_* \THH(\ell, D(v))
	\ar[rr]^-{1\otimes\partial_*} \ar[d]_-{\cong}^-{\cdot}
	&& P_{p-1}(u) \otimes V(1)_{*-1} \THH(\bZ_{(p)}) \ar[d]^{\chi_*} \\
V(1)_* \THH(\kup, D(u)) \ar[rr]^-{\partial'_*}
	&& V(1)_{*-1} \THH(\bZ_{(p)})
}
\]
of $P_{p-1}(u) \tensor V(1)_* \THH(\ell)$-modules, where $\chi_*(1 \otimes x) = x$ and
$\chi_*(u^k \otimes x) = 0$ for all $x \in V(1)_{*-1} \THH(\bZ_{(p)})$
and $k\ge1$.
Hence $\partial'_*(1 \cdot y) = \partial_*(y)$ and $\partial'_*(u^k
\cdot y) = 0$ for all $y \in V(1)_* \THH(\ell, D(v))$ and $k\ge1$.
\end{lemma}

\begin{proof}
It follows from Proposition~\ref{prop:Dv-Du-ell-ku-hty-cocartesian} and~\cite[Proposition~6.11]{RSS_LogTHH-I} that the map~$\chi$ may be obtained from the analogous map
\[
\bS^\cJ[D(u)] \wedge_{\bS^\cJ[D(v)]} \Sigma \bS^\cJ[B^{\cy}_{\{0\}} (D(v))]
	\longto \Sigma \bS^\cJ[B^{\cy}_{\{0\}}(D(u))] \,,
\]
by cobase change along \[\bS^{\cJ}[B^{\cy}(D(v))]\to \THH(\ell)\quad\text{ and }\quad\bS^{\cJ}[B^{\cy}(D(u))]\to \THH(\kup)\] (compare to the proof of Theorem~\ref{thm:ell-ku-Du-Dv-logthh-etale}). 
Clearly the $\bS^\cJ[D(u)]$-module structure on $\Sigma \bS^\cJ[B^{\cy}_{\{0\}}(D(u))]$
factors through the projection $\bS^\cJ[D(u)] \to \bS^\cJ[D(u)_{\{0\}}]$
to the $\cJ$-degree~$0$ part which implies that $\chi_*(u^k \otimes x) = 0$
for all $k\ge1$.  The remaining claims follow by naturality.
\end{proof}

Let $\tau' \colon \THH(\bZ_{(p)}) \to \THH(\kup)$ denote the homotopy fiber map of the map~$\rho'$ in \eqref{eq:cofseq-ludv-kudu}. As usual, we write $(u) = \bF_p\{u^k \mid 1 \le k \le p-2\}$ for the
ideal in $P_{p-1}(u)$ generated by $u$.

\begin{lemma}
There is a long exact sequence
\begin{multline*}
\dots \longto V(1)_* \THH(\bZ_{(p)}) \overset{\tau'_*}\longto
	V(1)_* \THH(\kup) \\
\overset{\rho'_*}\longto V(1)_* \THH(\kup, D(u))
	\overset{\partial'_*}\longto V(1)_{*-1} \THH(\bZ_{(p)})
	\longto \dots
\end{multline*}
of $V(1)_* \THH(\kup)$-modules, where $\rho'_*$ is an algebra
homomorphism. The resulting short exact sequence
\[
0 \to \ker(\rho'_*) \longto
	V(1)_* \THH(\kup) \longto \im(\rho'_*) \to 0
\]
is a square-zero extension of $P_{p-1}(u) \otimes V(1)_* \THH(\ell)$-algebras,
where
\[
\ker(\rho'_*) \cong E(\lambda_1) \otimes P(\mu_2) \{\lambda_2\}
\]
and
\[
\im(\rho'_*) = E(\lambda_1) \otimes P(\kappa_1^p)
\ \oplus\ %
(u) \otimes E(\lambda_1, d\log u) \otimes P(\kappa_1) \,.
\]
\end{lemma}

\begin{proof}
  By exactness $\im(\rho'_*) = \ker(\partial'_*)$, which by
  Lemma~\ref{lem:v1exseqelldv} and Lemma~\ref{lem:chipartial} is the
  direct sum of $\ker(\partial_*) = E(\lambda_1) \otimes
  P(\kappa_1^p)$ and $(u) \otimes E(\lambda_1, d\log u) \otimes
  P(\kappa_1)$, inside \[V(1)_* \THH(\kup, D(u)) = P_{p-1}(u) \otimes
  E(\lambda_1, d\log u) \otimes P(\kappa_1).\]

  Similarly $\ker(\rho'_*) \cong \cok(\partial'_*) = \cok(\partial_*)
  \cong \ker(\rho_*)$ equals $E(\lambda_1) \otimes P(\mu_2)
  \{\lambda_2\}$. This is a square-zero ideal inside $V(1)_*
  \THH(\ell) = E(\lambda_1, \lambda_2) \otimes P(\mu_2)$, hence is
  also a square-zero ideal inside $V(1)_* \THH(\kup)$.
\end{proof}

The following is essentially copied from \cite{Ausoni_THH-ku}*{Definition~9.13}.

\begin{definition}
Assume $p\ge3$, and let $\Theta_*$ be the graded-commutative unital
$P_{p-1}(u) \otimes P(\mu_2)$-algebra with generators
\[
\begin{cases}
a_i & 0 \le i \le p-1 \,, \\
b_j & 1 \le j \le p-1 \,,
\end{cases}
\]
and relations
\[
\begin{cases}
u^{p-2} a_i = 0 & 0 \le i \le p-2 \,, \\
u^{p-2} b_j = 0 & 1 \le j \le p-1 \,, \\
b_i b_j = u b_{i+j} & i+j \le p-1 \,, \\
a_i b_j = u a_{i+j} & i+j \le p-1 \,, \\
b_i b_j = u b_{i+j-p} \mu_2 & i+j \ge p \,, \\
a_i b_j = u a_{i+j-p} \mu_2 & i+j \ge p \,, \\
a_i a_j = 0 & 0 \le i, j \le p-1 \,.
\end{cases}
\]
By convention $b_0 = u$. The degrees of the generators
are $|a_i| = 2pi + 3$ and $|b_j| = 2pj + 2$.
\end{definition}

\begin{theorem}\cite{Ausoni_THH-ku}*{Theorem~9.15} \label{thm:ausoni}
Let $p\ge 3$.  There is an isomorphism
\[
V(1)_* \THH(\kup) \cong E(\lambda_1) \otimes \Theta_*
\]
of $P_{p-1}(u) \otimes E(\lambda_1) \otimes P(\mu_2)$-algebras.  Under
this identification,
$\rho'_*(a_i) = u \cdot d\log u \cdot \kappa_1^i$ 
	for $0 \le i \le p-1$,
$\rho'_*(b_j) = u \kappa_1^j$ for $1 \le j \le p-1$
and $\rho'_*(\mu_2) = \kappa_1^p$, all in $\im(\rho'_*)$,
and $\lambda_2$ in $\ker(\rho'_*)$ maps to
	a unit in $\bF_p$ times $u^{p-2} a_{p-1}$.
\end{theorem}

\begin{definition}
The assignments
\begin{align*}
\lambda_1 &\longmapsto \lambda_1 \\
a_i &\longmapsto u \cdot d\log u \cdot \kappa_1^i \\
b_j &\longmapsto u \kappa_1^j \\
\mu_2 &\longmapsto \kappa_1^p
\end{align*}
define a surjective algebra homomorphism $\bar\theta \: E(\lambda_1)
\otimes \Theta_* \to \im(\rho'_*)$, with kernel
$E(\lambda_1) \otimes P(\mu_2)\{u^{p-2} a_{p-1}\}$.
For $p\ge5$ it
lifts uniquely to an algebra homomorphism
\[
\theta \: E(\lambda_1) \otimes \Theta_* \longto V(1)_* \THH(\kup) \,,
\]
because $\ker(\rho'_*) = 0$ in the degrees of the algebra generators and
relations of $\Theta_*$.  (This is not true for $p=3$.)  For brevity,
let $z = u^{p-2} \kappa_1^{p-1}$.
\end{definition}

In the remainder of this section we will give a new proof of Ausoni's
theorem, in the cases $p\ge5$.  The only obstruction to carrying
out the same proof for $p=3$ is the need to check that $\bar\theta$
admits a multiplicative lift~$\theta$, i.e., that the relations in
degree~$|\lambda_2| = 2p^2-1 = 17$ and $|\lambda_1\lambda_2| = 2p^2+2p-2 =
22$ in $\Theta_*$ also hold in $V(1)_* \THH(ku_{(p)})$.  We do not carry
out this check.  A similar complication occurs in Ausoni's original
calculation for $p=3$, cf.~\cite[p.~1305]{Ausoni_THH-ku}.

\begin{lemma} \label{lem:theta_low}
For $p\ge5$, the homomorphism $\theta$ is an isomorphism in degrees $* <
|\lambda_2| = 2p^2-1$, and maps $\ker(\bar\theta) = E(\lambda_1) \otimes
P(\mu_2) \{u^{p-2} a_{p-1}\}$ to $\ker(\rho'_*) \cong E(\lambda_1)
\otimes P(\mu_2) \{\lambda_2\}$.  In degree~$2p^2-1$ it maps $u^{p-2}
a_{p-1}$ to the product $du \cdot z$, which is a multiple of $\lambda_2$.
It is an isomorphism in all degrees if and only if this multiple is nonzero.
\end{lemma}

\begin{proof}
The lift $\theta$ is an isomorphism if and only if it maps $\ker(\bar\theta)$
isomorphically to $\ker(\rho'_*)$, which happens if and only if $\theta$
maps $u^{p-2} a_{p-1} = a_0 \cdot u^{p-3} b_{p-1}$ to a nonzero multiple of
$\lambda_2$.  Since $a_0$ maps to $u \cdot d\log u = \rho'_*(du)$
and $u^{p-3} b_{p-1}$ maps to $z = u^{p-2} \kappa_1^{p-1}$, this
happens if and only if $du \cdot z \doteq \lambda_2$ in $V(1)_* \THH(\kup)$.
The only alternative is that $du \cdot z = 0$.
\end{proof}

\begin{proposition}
Consider the case $A = \kup$ and $M = D(u)$ of the spectral
sequence~\eqref{eq:logthhamss}.  It is an algebra spectral sequence
\begin{align*}
E^2_{**} &= \Tor^{H_\circledast (B^{\cy}(D(u))_{h\cJ})}_{**}
	(V(1)_* \THH(\kup), H_\circledast (B^{\rep}(D(u))_{h\cJ})) \\
&\Longrightarrow V(1)_* \THH(\kup, D(u)) \,,
\end{align*}
where
\begin{align*}
H_\circledast (B^{\cy}(D(u))_{h\cJ}) &= P(u) \otimes E(du) \otimes C_* \\
H_\circledast (B^{\rep}(D(u))_{h\cJ}) &= P(u) \otimes E(d\log u) \otimes C_* \\
V(1)_* \THH(\kup) &\cong E(\lambda_1) \otimes \Theta_* \\
V(1)_* \THH(\kup, D(u)) &= P_{p-1}(u) \otimes
	E(\lambda_1, d\log u) \otimes P(\kappa_1) \,.
\end{align*}
Here
\[
E^2_{**} = \Tor^{E(du)}_{**}
	(E(\lambda_1) \otimes \Theta_*, E(d\log u))
\]
is the tensor product of $E(\lambda_1, d\log u) \otimes P(\mu_2)$ with
\begin{multline*}
\bF_p \{ 1, u^{p-3} b_{p-1}, u^i b_j
	\mid 0 \le i \le p-4, 0 \le j \le p-1\}
\\ \oplus\ %
\Gamma([du]) \{u^{p-3} b_{j-1}, a_j \mid 1 \le j \le p-1\}
\end{multline*}
where $b_0 = u$.
There are nontrivial differentials
\[
d^2(\gamma_k[du] \cdot u^{p-3} b_{j-1}) \doteq \gamma_{k-2}[du] \cdot a_j
\]
modulo $(\lambda_1,d\log u),$ for all $k\ge2$ and $1 \le j \le p-1$, leaving $E^3_{**} = E^\infty_{**}$
equal to the tensor product of $E(\lambda_1, d\log u) \otimes P(\mu_2)$ with
\begin{multline*}
\bF_p \{ 1, u^{p-3} b_{p-1}, u^i b_j
        \mid 0 \le i \le p-4, 0 \le j \le p-1\}
\\ \oplus\ %
E([du]) \{u^{p-3} b_{j-1} \mid 1 \le j \le p-1\} \,.
\end{multline*}
Here $u$ represents $u$, $[du] u^{p-3} b_{j-1}$ represents $\kappa_1^j$
up to a unit in $\bF_p$, for $1 \le j \le p-1$, and there are
multiplicative extensions $u \cdot [du] u^{p-3} b_{j-1} \doteq b_j$
and $[du]u^{p-3}b_{j-1} \cdot [du]u^{p-3}b_{k-1} \doteq \mu_2$ for $1
\le j,k \le p-1$ with $j+k=p$.
\end{proposition}

\begin{proof}
We first rewrite the $E^2$-term to clarify its dependence on $V(1)_*
\THH(\kup)$, using change-of-rings:
\begin{align*}
E^2_{**} &= \Tor^{P(u) \otimes E(du) \otimes C_*}_{**}
	(V(1)_* \THH(\kup), P(u) \otimes E(d\log u) \otimes C_*) \\
&\cong \Tor^{E(du)}_{**}
	(V(1)_* \THH(\kup), E(d\log u)) \,.
\end{align*}
Here $E(du)$ acts trivially on $E(d\log u)$, so only the $E(du)$-module
structure on $V(1)_* \THH(\kup)$ is relevant.  We know that $\theta$ is
an isomorphism in degrees $* < |\lambda_2| = 2p^2-1 < |\mu_2| = 2p^2$,
so in this range of degrees $V(1)_* \THH(\kup)$ is additively isomorphic
to the tensor product of $E(\lambda_1)$ and the $E(du)$-module
\begin{multline*}
E(du) \{ 1, u^{p-3} b_{p-1}, u^i b_j \mid 0 \le i \le p-4, 0 \le j \le p-1 \}
\\ \oplus\ %
\bF_p \{ u^{p-3} b_{j-1}, a_j \mid 1 \le j \le p-1 \} \,,
\end{multline*}
with $u^{p-3} b_{p-1}$ corresponding to $z$ in degree~$2p^2-4$.

It follows that the $E^2$-term is as stated in the proposition in
bidegrees $(s,t)$ with $t < 2p^2-1$.  In particular, it is isomorphic
to a free module over $E(\lambda_1, d\log u)$ in this range of degrees.
Since the abutment is a free module over $E(\lambda_1, d\log u)$ on
the monomial generators of $P_{p-1}(u) \otimes P(\kappa_1)$, which
are concentrated in even degrees, it follows that the $E^2$-classes
$\gamma_i[du] \cdot a_j$ in (odd) total degrees less than $2p^2-1$ cannot
survive to the $E^\infty$-term.  By induction over increasing total
degrees~$s+t$, and over decreasing filtration degrees~$s$ within each total
degree, it follows that there must be nonzero $d^2$-differentials as
stated in the proposition, cancelling the $E(\lambda_1, d\log u)$-module
generators $\gamma_k[du] \cdot u^{p-3} b_{j-1}$ and $\gamma_{k-2}[du]
\cdot a_j$ for $k\ge2$ and $1\le j\le p-1$, in total degrees $s+t <
2p^2-1$.

If $du \cdot z \doteq \lambda_2$, so that $\theta$ is an isomorphism, the
same inductive argument continues to cover all total degrees, extending
linearly over $P(\mu_2)$.  The $E^\infty$-term is then concentrated
in filtration degrees $0 \le s \le 1$, and the final claims of the
proposition follow directly from a comparison with the known abutment.

It remains to exclude the alternative, namely that $du \cdot z = 0$.
In that hypothetical case, $V(1)_* \THH(\kup)$ would be isomorphic in
degrees $* \le 2p^2-1$ to the $E(du)$-module
\begin{multline*}
E(du) \{ 1, u^i b_j \mid 0 \le i \le p-4, 0 \le j \le p-1 \}
\\ \oplus\ %
\bF_p \{ z, \lambda_2, u^{p-3} b_{j-1}, a_j \mid 1 \le j \le p-1 \} \,.
\end{multline*}
This would lead to a modified $E^2$-term, where the summand $\bF_p
\{u^{p-3} b_{p-1}\}$ is replaced by $\Gamma([du]) \{z, \lambda_2\}$,
at least in internal degrees $t \le 2p^2-1$.

By our initial analysis for $t < 2p^2-1$, all $E^2$-generators in total
degree $s+t = 2p^2-2$ support linearly independent $d^2$-differentials.
Hence all generators in total degree $2p^2-1$ must be $d^2$-cycles.

Under this assumption, the $E^2$-generators in total degree $s+t = 2p^2-1$
would be $\lambda_2$ in filtration $s=0$, the $m = (p-1)/2$ classes
\[
\{ \gamma_{pi-1}[du] \cdot a_{p-2i} \mid 1 \le i \le m \} \,,
\]
and some $\lambda_1$- or $d\log u$-multiples of classes in lower total
degrees.

The $E^2$-generators in total degree $s+t = 2p^2$ and filtration
degree $s\ge2$ would be the $m$ classes
\[
\{ \gamma_{pi+1}[du] \cdot u^{p-3} b_{p-2i-1} \mid 1 \le i \le m \} \,,
\]
and some $\lambda_1$- or $d\log u$-multiples of classes in lower total
degrees.

By our previous analysis, the $\lambda_1$- or $d\log u$-multiples in total
degree $2p^2$ and filtration degree $\ge2$ support $d^2$-differentials
that kill all but one of the $\lambda_1$- or $d\log u$-multiples in
total degree $2p^2-1$, leaving only $\lambda_1 [du] u^{p-3} b_{p-2}$
in filtration degree~$s=1$.  This is $\lambda_1$ times the permanent
cycle $[du] u^{p-3} b_{p-2}$ representing $\kappa_1^{p-1}$.

The only remaining $d^r$-differentials affecting total degree $2p^2-1$,
for $r\ge2$, are those mapping from the $m$ classes $\gamma_{pi+1}[du]
\cdot u^{p-3} b_{p-2i-1}$ to the $m+2$ classes $\gamma_{pi-1}[du]
\cdot a_{p-2i}$, $\lambda_2$ and $\lambda_1 [du] u^{p-3} b_{p-2}$.
It follows that at least $(m+2)-m = 2$ linearly independent classes are
left at the $E^\infty$-term in total degree~$2p^2-1$.  This contradicts
the fact that the abutment in degree~$2p^2-1$ is generated by the single
class $\lambda_1 \kappa_1^{p-1}$.

This contradiction eliminates the possibility that $du \cdot z = 0$.
Hence $\theta$ is an isomorphism in all degrees, and the structure
of the spectral sequence is as asserted.
\end{proof}

\begin{proof}[Proof of Theorem~\ref{thm:ausoni} for $p\ge5$]
In view of Lemma~\ref{lem:theta_low}, and the conclusion from the proof
of the previous proposition that $du \cdot z \doteq \lambda_2$,
the claims of the theorem have now been verified for $p\ge5$.
\end{proof}

\begin{proposition}\phantomsection \label{prop:circle-operator-values}
\begin{enumerate}[(i)]
\item In $V(1)_* \THH(\ell) \cong E(\lambda_1, \lambda_2) \otimes
P(\mu_2)$ the suspension operator satisfies $\sigma(\lambda_1) = 0$,
$\sigma(\lambda_2) = 0$ and $\sigma(\mu_2) = 0$. 

\item In $V(1)_* \THH(\ell, D(v)) \cong E(\lambda_1, d\log v) \otimes
P(\kappa_1)$ the suspension operator satisfies $\sigma(\lambda_1) = 0$,
$\sigma(d\log v) = 0$ and $\sigma(\kappa_1) = \kappa_1 \cdot d\log v$.

\item In $V(1)_* \THH(ku_{(p)}) \cong E(\lambda_1) \otimes \Theta_*$ the
suspension operator satisfies $\sigma(\lambda_1) = 0$, $\sigma(a_i) =
0$, $\sigma(b_j) = (1-j) a_j$ and $\sigma(\mu_2) = 0$.

\item In $V(1)_* \THH(ku_{(p)}, D(u)) \cong P_{p-1}(u) \otimes E(\lambda_1,
d\log u) \otimes P(\kappa_1)$ the suspension operator satisfies $\sigma(u)
= u \cdot d\log u$, $\sigma(\lambda_1) = 0$, $\sigma(d\log u) = 0$
and $\sigma(\kappa_1) = - \kappa_1 \cdot d\log u$.
\end{enumerate}
\end{proposition}

\begin{proof}
\textit{(i)} The Hurewicz image of $\lambda_2$ in $H_*(V(1) \wedge
\THH(\ell))$ is $\sigma\bar\xi_2$, hence $\sigma(\lambda_2) = 0$. The
classes $\sigma(\lambda_1)$ and $\sigma(\mu_2)$ are zero because they
lie in trivial groups.

\textit{(iv)}   We saw that $\sigma(u) = u \cdot d\log u$ and $\sigma(d\log
u) = 0$ in Theorem~\ref{thm:thhkudu}.  The class $\sigma(\lambda_1)$
is zero by case~\textit{(i)}, via naturality with respect to
$\THH(\ell) \to \THH(ku_{(p)}, D(u))$.
Under the trace map,  Ausoni's class $b \in V(1)_{2p+2} K(ku)$ is mapped to the class $b_1 \in V(1)_{2p+2} \THH(ku_{(p)})$. Hence $\sigma(b_1) = 0$.
This can also be deduced from the formula for
Connes' $B$-operator; compare~\cite{Ausoni_THH-ku}*{Remark~3.4}
and~\cite{Ausoni_Kku}*{Lemma~6.3}.  Since $\rho'_*(b_1) = u\kappa_1$
it follows that $0 = u \cdot d\log u \cdot \kappa_1 + u \cdot
\sigma(\kappa_1)$ and $\sigma(\kappa_1) = - \kappa_1 \cdot d\log u$.

\textit{(ii)} This follows from case~\textit{(iv)} by naturality with respect to
the morphism $(f, f^\flat) \colon (\ell, D(v)) \to (ku_{(p)}, D(u))$.

\textit{(iii)} This follows from case~\textit{(iv)} by naturality with respect to the morphism 
$\rho' \colon \THH(ku_{(p)}) \to \THH(ku_{(p)}, D(u))$.
\end{proof}


\end{document}